\documentclass[a4paper,12pt,leqno]{amsart} 
\usepackage{amsmath,amsthm,amssymb}  
\usepackage{mathdots}

\usepackage{graphicx}

\numberwithin{equation}{section}

\theoremstyle{plain}
\newtheorem{theorem}{Theorem}[section]
\newtheorem{proposition}[theorem]{Proposition}         
\newtheorem{corollary}[theorem]{Corollary} 
\newtheorem{lemma}[theorem]{Lemma} 
   
\theoremstyle{definition}  
 
\newtheorem{remark}[theorem]{Remark} 

\newcommand{\C}{\mathbb C}   
\newcommand{\R}{\mathbb R}
\newcommand{\Z}{\mathbb Z}

\renewcommand{\P}{\mathbb P}
\newcommand{\X}{\mathbb X}  

\newcommand{\al}{\alpha}
 
\newcommand{\ga}{\gamma}
\newcommand{\de}{\delta}
\newcommand{\la}{\lambda}
\newcommand{\si}{\sigma} 
 
\newcommand{\La}{\Lambda}

\newcommand{\Om}{\Omega}
\newcommand{\De}{\Delta}
\renewcommand{\th}{\theta}
\newcommand{\om}{\omega}

\newcommand{\ze}{\zeta}

\DeclareMathOperator{\tr}{tr}

\DeclareMathOperator{\diag}{diag}

\DeclareMathOperator{\ar}{arg}
\DeclareMathOperator{\maxx}{max}

\renewcommand{\Re}{\mathrm{Re\,}}

\newcommand{\no}{\noindent}

\newcommand{\pr}{\prime} 
\newcommand{\prr}{{\prime\prime}} 
\newcommand{\st}{\ \vert\ }   
\renewcommand{\ll}{\lq\lq}
\newcommand{\rr}{\rq\rq\ }
\newcommand{\rrr}{\rq\rq}  
\renewcommand{\b}{\partial}
\newcommand{\zbar}{  {\bar z}  }
\newcommand{\zzb}{ {z\bar z}  }

\renewcommand{\smallint}{  {\textstyle \text{$\int$} }}
\renewcommand{\i}{ {\scriptscriptstyle\sqrt{-1}}\, }

\newcommand{\bp}{\begin{pmatrix}} 
\newcommand{\ep}{\end{pmatrix}} 

\newcommand{\four}{\vphantom{a^a} 4} 
\newcommand{\five}{\vphantom{a^a} 5} 
\newcommand{\six}{\vphantom{a^a} 6}

\newcommand{\SL}{\textrm{SL}}

\newcommand{\SO}{\textrm{SO}}

\newcommand{\SU}{\textrm{SU}}

\renewcommand{\sl}{\frak s\frak l}

\newcommand{\sss}{\scriptsize}

\begin{document}     

\title[Isomonodromy aspects I: Stokes data]{Isomonodromy aspects of the tt*
equations of Cecotti and Vafa 
\\
I. Stokes data}  

\author{Martin A. Guest, Alexander R. Its, and Chang-Shou Lin}      

\date{}   

\begin{abstract} We describe all smooth solutions 
of the two-function tt*-Toda equations (a version of the tt* 
equations,  or equations for harmonic maps into $\SL_n\R/\SO_n$) in terms
of (i) asymptotic data, (ii) holomorphic data, and (iii)
monodromy data.  This allows us to find all solutions with integral Stokes data. 
These include solutions associated to nonlinear sigma models (quantum cohomology) or Landau-Ginzburg models (unfoldings of singularities), as conjectured by Cecotti and Vafa.
In particular we establish the existence of a new family of pure and polarized TERP structures in the sense of \cite{He03}, or noncommutative variations of Hodge structures in the sense of \cite{KaKoPa07}.
\end{abstract}


\maketitle 

\section{Introduction}\label{intro}

\subsection{The p.d.e.\ and the monodromy data}\label{first}

The two-dimensional Toda lattice is an important integrable system with many aspects. 
It is an example of a nonabelian Chern-Simons theory in classical field theory (see \cite{YiYa01}),
and in differential geometry it can be interpreted as the equation
for primitive harmonic maps taking values in a compact flag manifold  (\cite{BoPeWo95},\cite{BuPe94}).  Such maps are closely related to  harmonic maps into symmetric spaces.  These in turn have many geometrical interpretations, e.g.\ 
surfaces in $\R^3$ of constant mean curvature (see \cite{Do08}), or
special Lagrangian cones in $\C^3$ 
(\cite{Jo08},\cite{Mc03}). 

We shall be concerned with the \ll two-dimensional periodic Toda lattice with opposite sign\rrr.
This is the system
\begin{equation}\label{ost}
 2(w_i)_{\zzb}=-e^{2(w_{i+1}-w_{i})} + e^{2(w_{i}-w_{i-1})}, \ \  
 w_i:U\to\R
\end{equation}
where $U$ is some open subset of $\C=\R^2$. We assume that $w_i=w_{i+n+1}$ for all $i\in\Z$ and  $w_0+\cdots+w_n=0$. 

The first manifestation of the integrability of this system is its zero curvature formulation:
\[
\text{$d\al+\al\wedge\al=0$ \ \ for all $\la\in\C^\ast=\C-\{0\}$} 
\]
where
\[
 \al(\lambda)=\al^\pr(\lambda) dz + \al^\prr(\lambda) d\zbar
 = (w_z+\tfrac1\la W^t)dz + (-w_{\zbar}+\la W)d\zbar,
 \]
 and
 \[
 w=\diag(w_0,\dots,w_n),\ \ 
 W=
 \left(
\begin{array}{c|c|c|c}
\vphantom{(w_0)_{(w_0)}^{(w_0)}}
 & \!e^{w_1\!-\!w_0}\! & &  
 \\
\hline
  &  & \  \ddots\   & \\
\hline
\vphantom{\ddots}
  & &  &  e^{w_n\!-\!w_{n\!-\!1}}\!\!\!
\\
\hline
\vphantom{(w_0)_{(w_0)}^{(w_0)}}
\!\! e^{w_0\!-\!w_n} \!\!  & &  &  \!
\end{array}
\right).
\]
In other words, (\ref{ost}) is the
compatibility condition $2w_{z\zbar}=[W^t,W]$ for the linear system
\begin{align}\label{*}
\begin{cases}
\Psi_z&=(w_z+\tfrac1\la W)\Psi\\
\Psi_\zbar&=(-w_{\zbar}+\la W^t)\Psi.
\end{cases}
\end{align}
The fact that $\al^\pr,\al^\prr$ involve only $1/\la,\la$ leads to holomorphic data for solutions of (\ref{ost}).  We recall from section 4 of \cite{GuLiXX} that
the holomorphic data is a matrix of the form
\[
\eta=
\bp
 & & & p_0\\
 p_1 & & & \\
  & \ddots & & \\
   & & p_n &
   \ep
\]
where each $p_i=p_i(z)$ is a holomorphic function. 

For radial solutions of (\ref{ost}), i.e.\ when
$w_i=w_i(z,\zbar)$ depends only on the real variable $x=\vert z\vert$,  (\ref{*}) reduces to
\begin{equation}\label{xequation}
\Psi_x= \tfrac1x\left( z \Psi_z + \zbar  \Psi_\zbar\right)
= \left( \tfrac 1{\la}\tfrac zx W + \la \tfrac{\zbar} x  W^t \right)\Psi.
\end{equation} 
This can be regarded as the equation for an isomonodromic deformation, which is another --- perhaps more famous --- manifestation of integrability, and a well known approach  to studying equations of Painlev\'e type (see \cite{FlNe80},\cite{JMU} and also \cite{FIKN06}).  

Namely, if we impose 
the (Euler-type) homogeneity condition
\begin{equation}\label{ho}
\la\Psi_\la + z\Psi_z-\zbar\Psi_\zbar = 0,
\end{equation}
we obtain
\begin{equation}\label{lambdaequation}
\Psi_\la =  \left( -\tfrac{1}{\la^2} zW - \tfrac{1}{\la} xw_x + \zbar W^t \right) \Psi
\end{equation}
which is a meromorphic o.d.e.\ in $\la$ with poles of order two at 
$0$ and $\infty$.   Writing $\mu=\la x/z$, we obtain
\begin{equation}\label{muequation}
\Psi_\mu =  \left( -\tfrac{1}{\mu^2} xW - \tfrac{1}{\mu} xw_x + xW^t\right) \Psi.
\end{equation}
The  compatibility condition of the linear system
\begin{align}\label{**}
\begin{cases}
\Psi_\mu&=\left( -\tfrac{1}{\mu^2} xW - \tfrac{1}{\mu} xw_x + xW^t\right) \Psi\\
\Psi_x&=\left(  \tfrac1\mu W + \mu W^t\right)\Psi
\end{cases}
\end{align}
is the radial version
$(xw_{x})_x=2x[W^t,W]$) of (\ref{ost}).  

From the connection point of view, the homogeneity assumption (\ref{ho}) extends the flat connection 
$d+\al$ to a flat connection 
$d+\al+\hat\al$, where 
$\hat\al(\mu) =  \left( -\tfrac{1}{\mu^2} xW - \tfrac{1}{\mu} xw_x + xW^t\right)d\mu$.   
It is well known that in this situation the $\mu$-system (\ref{muequation})  is isomonodromic, i.e.\ its monodromy data is independent of $x$.  Conversely, starting from (\ref{muequation}),  one may seek its  isomonodromic deformations (see \cite{JMU}), and this leads back to (\ref{xequation}). 

The monodromy data of (\ref{muequation}) consists of formal monodromy and Stokes matrices at each of the two poles, as well as a connection matrix relating them.   Locally, solutions of the radial version of (\ref{ost}) correspond to such data. However this kind of data is in general difficult to compute explicitly, and global properties of the corresponding solution are difficult to read off.

We shall impose the following \ll anti-symmetry condition\rr
\begin{equation}\label{as}
\begin{cases}
\ \  w_0+w_{l-1}=0, \ w_1+w_{l-2}=0,\ \ \dots\\
\ \  w_l+w_{n}=0, \ w_{l+1}+w_{n-1}=0,\ \ \dots
 \end{cases}
\end{equation}
for some $l\in\{0,\dots,n+1\}$. For $l=0$ (equivalently $l=n+1$), this means that
$w_i+w_{n-i}=0$ for all $i$;  in this case (\ref{ost}) and (\ref{as}) is the system studied by Cecotti and Vafa in \cite{CeVa91},\cite{CeVa92},\cite{CeVa92a}.  In general we call (\ref{ost}) and (\ref{as}) together with the radial assumption the {\em tt*-Toda equations.}

In this article we shall identify all solutions of the tt*-Toda equations which are smooth on the \ll maximal\rr domain $U=\C^\ast$. They constitute a $2$-dimensional family.  We shall give simple and explicit parametrizations of this family in terms of asymptotic data, holomorphic data, and monodromy data.  In section \ref{interp} we explain the important geometrical and physical motivation behind this description.

To end this section, we mention that equations (\ref{*}) and (\ref{as}) possess the following symmetries, which will play an important role in our computation of the monodromy data.  They
depend on the three automorphisms $\tau, \si, c$ of 
\[
\sl_{n+1}\C=\{
\text{complex $(n+1)\times(n+1)$ matrices $X$ with $\tr X=0$}
\}
\]
which are defined by
\[
\tau(X)=d_{n+1}^{-1} X d_{n+1},\ \ \si(X)=-\De\, X^t\,  \De,\ \ c(X)=\De \bar X  \De
\]
where
$d_{n+1}=\diag(1,\om,\dots,\om^n)$,  $\om=e^{{2\pi \i}/{(n+1)}}$,
and
\[
\De=\De_{l,n+1-l}=
\bp
J_l & \\
 & J_{n+1-l}
 \ep,
 \quad
 J_l=
 \bp
  & & 1 \\
  & \iddots \, & \\
1 & &
  \ep  \ \text{($l\times l$ matrix).}
 \]
Namely, it is easily verified that the connection form $\al$ has the following properties.

\no{\em Cyclic symmetry: }  $\tau(\al(\la))=\al(e^{{2\pi \i}/{(n+1)}} \la)$

\no{\em Anti-symmetry: }  $\si(\al(\la))=\al(-\la)$

\no{\em Reality: }  $c(\al^\pr(\la))=\al^\prr(1/\bar\la)$

\no The cyclic property expresses the shapes of the matrices in (\ref{*}), the anti-symmetry property is just condition (\ref{as}), and the reality property corresponds to the fact that every $w_i$ is real.

\subsection{Geometry}\label{interp}

The differential geometric or Lie-theoretic meaning of the three symmetries, and the proof of the following result, can be found in section 2 of our previous article  \cite{GuLiXX}.

\begin{proposition} From any solution
of the system (\ref{ost}) and (\ref{as})
we obtain a harmonic map $U\to \SL_{n+1}\R/\SO_{n+1}$.  
\end{proposition} 

More generally, it was shown in \cite{Du93} (see also \cite{CoSc05}) that the tt* equations are --- locally --- the equations for (pluri)harmonic maps into 
$\SL_{n+1}\R/\SO_{n+1}$, together with a homogeneity condition (which amounts to the radial condition in our situation).  This is the basis for the geometric interpretation of the tt* equations, as $\SL_{n+1}\R/\SO_{n+1}$ is the space of 
inner products on $\R^{n+1}$.  A solution of the tt* equations can thus be regarded as a (trivial) vector bundle of rank $n+1$ on $U$ equipped with a metric satisfying various natural conditions.  In fact it is a (trivial) harmonic bundle in the sense of \cite{Si92} with an additional real structure and homogeneity property.

Such bundles (not necessarily trivial) have been well studied.  They generalize variations of polarized Hodge structures, and they appear in 
\cite{Bar01},\cite{He03},\cite{KaKoPa07} in connection with semi-infinite variations of Hodge structures, pure polarised TERP structures, or noncommutative variations of Hodge structures.   A special role is played by those bundles which arise \ll from geometry\rrr, e.g.\ from variations of polarized Hodge structures on specific manifolds.   

The relation with the discussion of the previous section is as follows. 
The \ll magical\rr solutions of the tt* equations predicted by Cecotti and Vafa are expected to be

\no(1) globally defined on $\C^\ast$,

\no(2) characterized by initial/asymptotic conditions,

\no(3) and have integral Stokes data.

\no We shall refer to such solutions as \ll field-theoretic\rrr.  Many of them appear to be of algebro-geometric origin, i.e.\ the corresponding field theories are constructed from unfoldings of singularities (Landau-Ginzburg models) or from quantum cohomology (nonlinear sigma models). 
 
Cecotti and Vafa predicted the existence of the global solutions on the basis of the physical interpretation of the monodromy data and holomorphic data:  
the Stokes data counts solitons at the \ll infra-red point\rr $z=\infty$ and the holomorphic data encodes chiral charges at the \ll ultra-violet point\rr $z=0$.  These are fixed points for 
the renormalization group flow, and it is this flow that is governed by the tt* equations.   For the \ll physical\rr solutions the data at $z=0$ and $z=\infty$ should have integrality properties.

Thus, it has been predicted that the tt* equations admit certain  \ll globally smooth\rr solutions with extremely rich geometrical meaning.  Our results confirm this prediction in the case where $w_0,\dots,w_n$ reduce to two independent functions (in this situation $n=3$, $4$, or $5$). First, in this case, we describe all solutions of  the tt*-Toda equations satisfying property (1).  The fact that such solutions are in one to one correspondence with asymptotic data at $z=0$ can be proved by using a monotone iteration scheme based on \cite{GuLiXX}.  This technique is well known for nonlinear scalar p.d.e., but novel for systems.  
Then, by computing the monodromy data and holomorphic data explicitly, we shall verify properties (2), (3).  
The solutions with integer Stokes data will be discussed in detail in \cite{GuLi2XX}). 

We remark that the holomorphic data is familiar to
differential geometers as the generalized Weierstrass representation, or DPW representation, of a harmonic map.  The link between the p.d.e.\ solution and the holomorphic data is given by the Iwasawa factorization $L=FB$ of a certain loop group valued function $L$ (with $L^{-1}dL=\tfrac1\la \eta dz$).  This is explained in detail in section 4 of \cite{GuLiXX}.
It is here that the crucial difference between the Toda equations and the tt*-Toda equations can be seen.   For the Toda equations, it is easy to construct solutions with properties (1) and (2),  because the relevant Lie group is the compact group $\SU_{n+1}$, and in this case  the Iwasawa factorization $L=FB$ holds on the entire domain of $L$. In contrast, for
the tt*-Toda equations we have the noncompact group $\SL_{n+1}\R$, and the Iwasawa factorization $L=FB$ holds (in general) only on some open subset $U^\pr$ of the domain $U$ of $L$.  Proving that $U^\pr=U$ is equivalent to proving that the solution of the p.d.e.\ has no singularities, and this is what we shall do.
The same factorization problem arises in the construction of other types of harmonic maps with important differential geometric interpretations (see, for example, \cite{Do08}).  As few examples are known where global existence can be proved in the noncompact case, our examples are of interest in this wider context.

A detailed statement of results is given in section \ref{results}.
In section \ref{pdetheorem} we prove the basic existence theorem (Theorem A),
and in section \ref{stokes} we compute the monodromy data of the solutions (Theorem B).
All methods used in this article are \ll low-tech\rrr, and our proofs are essentially self-contained.  We hope that this will contribute to a better understanding of the tt* equations, which have been treated until now only by indirect methods.

This article can be read independently of \cite{GuLiXX}, although our proof of Theorem A will make reference to \cite{GuLiXX} in order to avoid repetition.   In \cite{GuLiXX} we constructed a family of globally smooth solutions of the tt*-Toda equations, and we identified a finite number of \ll field-theoretic solutions\rr amongst them.  In this article we complete the picture by constructing {\em all} globally smooth solutions, and we describe {\em all} field-theoretic solutions.   We do this 
only in the case where $w_0,\dots,w_n$ are equivalent to two unknown functions; the general case will be treated in a subsequent article.
In \cite{GuItLi2XX} we give an independent proof of the existence of solutions, from the isomonodromy point of view.

Acknowledgements: 
The first author was partially supported by a grant from the JSPS, and the second author by NSF grant 
DMS-1001777.  Both are grateful to Taida Institute for Mathematical Sciences for financial support and hospitality during their visits in 2011, when much of this work was done.   The authors thank Claus Hertling for his comments on an earlier version of the paper, and also Takuro Mochizuki for his interest in this project.  Mochizuki has given a proof of our main results from the viewpoint of the Hitchin-Kobayashi correspondence in his preprint \ll Harmonic bundles and Toda lattices with opposite sign\rrr,
arXiv:\  1301.1718.

\section{Results}\label{results}

For $a,b>0$, we consider the system
\begin{equation}\label{pde}
\begin{cases}
u_{z\zbar}&= \ e^{au} - e^{v-u} 
\\
v_{z\zbar}&= \ e^{v-u} - e^{-bv}
\end{cases}
\end{equation}
for $u,v:\C^\ast\to\R$.
This is equivalent to (\ref{ost}) and (\ref{as}) in the ten cases where 
$w_0,\dots,w_n$ reduce to two unknown functions  (and in such cases we have $a,b\in\{1,2\}$).  This data is given in Table \ref{tableofcases}.  For example, in the case labelled 4a, condition (\ref{as}) says that $w_0+w_3=0$ and $w_1+w_2=0$, 
so the tt*-Toda equations reduce to (\ref{pde})
with $a=b=2$ if we put $u=2w_0$, $v=2w_1$.  

Evidently the cases $(a,b)=(1,2)$ and $(2,1)$ correspond under the transformation $(u,v)\mapsto(-v,-u)$, so there are in fact just three distinct p.d.e.\ systems.

\begin{table}[h]
\renewcommand{\arraystretch}{1.3}
\begin{tabular}{c||cc|cc|cc}\label{cases}
label & $l$ & $n\!\!+\!\!1\!\!-\!l$ & $u$ & $v$ & $a$ & $b$
\\
\hline
4a & $4$ & 0 & $2w_0$ & $2w_1$ & 2 & 2  
\\
4b & $2$ & $2$ & $2w_3$ & $2w_0$ & 2 & 2
\\
\hline
5a & $5$ & 0 & $2w_0$ & $2w_1$ & 2 & 1 
\\
5b & $3$ & $2$ & $2w_4$ & $2w_0$ & 2 & 1
\\
\hline
5c & $4$ & $1$ & $2w_0$ & $2w_1$ & 1 & 2 
\\
5d & $1$ & $4$ & $2w_1$ & $2w_2$ & 1 & 2  
\\
5e & $2$ & $3$ & $2w_4$ & $2w_0$ & 1 & 2 
\\
\hline
6a & $5$ & $1$ & $2w_0$ & $2w_1$ & 1 & 1 
\\
6b & $1$ & $5$ & $2w_1$ & $2w_2$ & 1 & 1 
\\
6c & $3$ & $3$ & $2w_5$ & $2w_0$ & 1 & 1   
\end{tabular}
\bigskip
\caption{}\label{tableofcases}
\end{table}

\no{\bf Theorem A:} {\em
Let $a,b>0$.   For any $(\ga,\de)$ in the triangular region 
\[
\text{$\ga\ge -2/a$, $\de\le 2/b$, $\ga-\de\le2$}
\]
(Fig.\ \ref{theregion})
the system (\ref{pde})
has a unique solution $(u,v)$ such that
\begin{gather*}
\ u(z) \to 0,
\ v(z) \to 0\ 
\text{as}\ \vert z\vert\to \infty
\\
\ u(z) = (\ga+o(1)) \log \vert z\vert,
\ v(z) = (\de+o(1)) \log \vert z\vert\
\text{as}\ \vert z\vert\to 0.
\end{gather*}
When $(\ga,\de)$ is in the interior of the region, we have
\[
\ u(z)=\ga\log\vert z\vert + O(1),
\ v(z)=\de\log\vert z\vert + O(1)\ 
\text{as}\ \vert z\vert\to 0.
\]
Since rotation of the variable $z$ does not affect these asymptotic conditions,  the functions $u,v$ depend only on $\vert z\vert$. 
}

\begin{figure}[h]
\begin{center}
\includegraphics[scale=0.6, trim= 0 300 0 250]{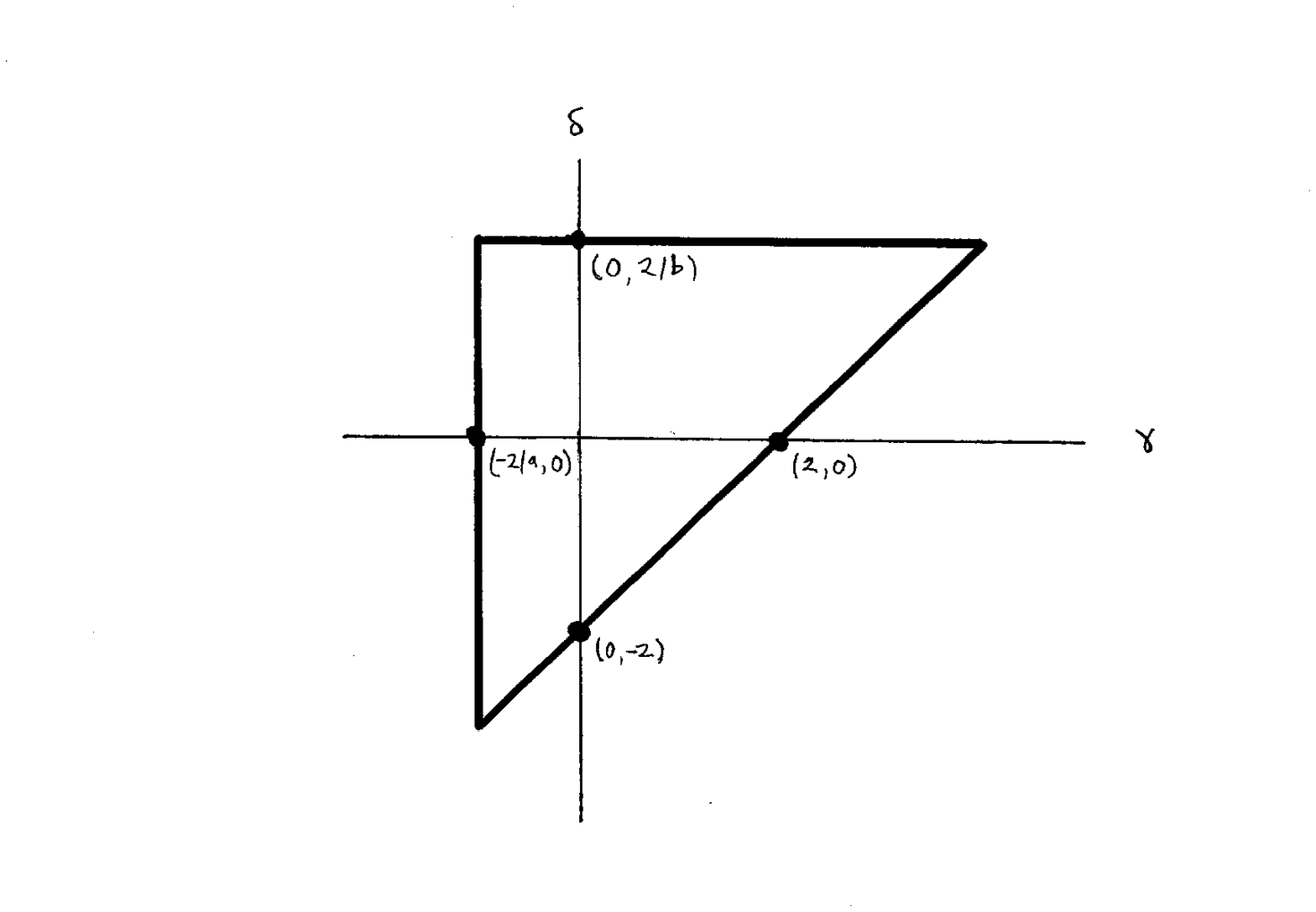}
\end{center}
\caption{}\label{theregion}
\end{figure}

\begin{remark}\label{radialasymptotics}  It should be noted that we do not assume $a,b\in\{1,2\}$ here, nor do we assume (a priori) that the solutions are radial.  In 
an appendix to \cite{GuItLi2XX}
we 
show that any radial solution which is smooth on $\C^\ast$ must satisfy the above asymptotic conditions at $z=0, \infty$.  Thus, Theorem A accounts for {\em all} radial solutions which are smooth on $\C^\ast$.
\end{remark}

\begin{remark} The asymptotic conditions at zero  can be written in the form 
$2w_i(z) = (\ga_i+o(1)) \log \vert z\vert$, $0\le i\le n$, 
if we define $\ga_i$ in terms of $\ga,\de$ according to Table \ref{tableofcases}.  For example, in case 4b, we define $\ga_0=\de$, $\ga_1=-\de$,$\ga_2=-\ga$,$\ga_3=\ga$.
\end{remark}

\begin{remark} 
For the case $a=b=2$, these solutions belong
to the class of solutions  of the periodic 2D
cylindrical Toda equations  which was introduced in
\cite{Wi97} via explicit  Fredholm  determinant  formulae. These formulae were then used in \cite{TrWi98} in order to analyze the
asymptotic behaviour of the solutions at $0$ and 
$\infty$.  We shall discuss the relation of \cite{Wi97} and \cite{TrWi98} with our work in a subsequent article.
\end{remark}

For each such solution, the Stokes data of the associated meromorphic o.d.e.\
reduces to two real numbers $s_1^\R,s_2^\R$, and we have:

\no{\bf Theorem B:} {\em  For each solution appearing in Theorem A, the Stokes data $s_1^\R,s_2^\R$ is given in terms of the asymptotic data $(\ga,\de)$ as follows.

\no (i) Cases 4a, 4b:

\no$\pm s_1^\R = 2\cos \tfrac\pi4 {\scriptstyle (\ga+1)} +  2\cos \tfrac\pi4 {\scriptstyle(\de+3)}$

\no$-s_2^\R = 2+4\cos \tfrac\pi4 {\scriptstyle(\ga+1)} \, \cos \tfrac\pi4 {\scriptstyle(\de+3)}$

\no (ii) Cases 5a, 5b:

\no\ \ $s_1^\R = 1+2\cos \tfrac\pi5 {\scriptstyle(\ga+6)}  + 2\cos \tfrac\pi5 {\scriptstyle(\de+8)}$

\no$-s_2^\R = 2
+2\cos \tfrac\pi5 {\scriptstyle(\ga+6)}  +  2\cos \tfrac\pi5 {\scriptstyle(\de+8)}
+4\cos \tfrac\pi5 {\scriptstyle(\ga+6)} \, \cos \tfrac\pi5 {\scriptstyle(\de+8)}$

\no (iii) Cases 5c, 5d, 5e:

\no
\ \ $s_1^\R = 1+2\cos \tfrac\pi5 {\scriptstyle(\ga+2)}  +  2\cos \tfrac\pi5 {\scriptstyle(\de+4)}$

\no$-s_2^\R = 2
+2\cos \tfrac\pi5 {\scriptstyle(\ga+2)}  +  2\cos \tfrac\pi5 {\scriptstyle(\de+4)}
+4\cos \tfrac\pi5 {\scriptstyle(\ga+2)} \, \cos \tfrac\pi5 {\scriptstyle(\de+4)}$

\no (iv) Cases 6a, 6b, 6c:

\no
$\pm s_1^\R = 2\cos \tfrac\pi6 {\scriptstyle(\ga+2)}  +  2\cos \tfrac\pi6 {\scriptstyle(\de+4)}$

\no$-s_2^\R = 1+4\cos \tfrac\pi6 {\scriptstyle(\ga+2)} \, \cos \tfrac\pi6 {\scriptstyle(\de+4)}$
}

\no For the even-dimensional cases, our calculation does not determine the sign of $s_1^\R$.  These signs are determined in \cite{GuItLi2XX}.  It is then an elementary matter to see that the correspondence between $(\ga,\de)$ and $(s_1^\R,s_2^\R)$ in Theorem B is bijective.  The sign is irrelevant for the classification of solutions in the $(\ga,\de)$ region with integral Stokes data, as the sign change corresponds to the involution $(\ga,\de)\mapsto(-\de,-\ga)$ of this region. These solutions are listed in Table \ref{integralsolutions}.

\begin{table}[h]
\renewcommand{\arraystretch}{1.5}
\begin{tabular}{c|c|c|c}
Cases 4a,4b & Cases 5a,5b & Cases 5c,5d,5e & Cases 6a,6b,6c
 \\
\hline
\sss $(3,1)$ & \sss $(4,2)$ & \sss $(3,1)$ & \sss $(4,2)$
\\
[-6pt]
\sss $(\tfrac{5}{3},1)$ & \sss $(\tfrac{7}{3},2)$ & \sss $(\tfrac{4}{3},1)$ & \sss $(2,2)$
\\
[-6pt]
\sss $(1,1)$ &  \sss $(\tfrac{3}{2},2)$ & \sss $(\tfrac{1}{2},1)$ &  \sss $(1,2)$
\\
[-6pt]
\sss $(\tfrac{1}{3},1)$ &  \sss $(\tfrac{2}{3},2)$ &  \sss $(-\tfrac{1}{3},1)$ &  \sss $(0,2)$
\\
[-6pt]
\sss $(-1,1)$ &  \sss $(-1,2)$ & \sss $(-2,1)$ &  \sss $(-2,2)$
\\
\hline
\sss $(-1,-\tfrac13)$ & \sss $(-1,\tfrac13)$ &  \sss $(-2,-\tfrac23)$ & \sss  $(-2,0)$
\\
[-6pt]
\sss  $(-1,-1)$ &  \sss $(-1,-\tfrac12)$ &  \sss $(-2,-\tfrac32)$ &  \sss $(-2,-1)$
 \\
 [-6pt]
\sss $(-1,-\tfrac53)$ &  \sss $(-1,-\tfrac43)$ &  \sss $(-2,-\tfrac73)$ &  \sss $(-2,-2)$
 \\
 [-6pt]
 \sss $(-1,-3)$ &  \sss $(-1,-3)$ &  \sss $(-2,-4)$ &  \sss $(-2,-4)$
 \\
\hline
\sss $(\tfrac{1}{3},-\tfrac{5}{3})$ & \sss $(\tfrac{2}{3},-\tfrac{4}{3})$ & \sss $(-\tfrac{1}{3},-\tfrac{7}{3})$ &  \sss $(0,-2)$
 \\
 [-6pt]
\sss $(1,-1)$ &  \sss $(\tfrac32,-\tfrac12)$ &  \sss $(\tfrac{1}{2},-\tfrac{3}{2})$ &  \sss $(1,-1)$
 \\
 [-6pt]
\sss $(\tfrac{5}{3},-\tfrac{1}{3})$ &  \sss $(\tfrac{7}{3},\tfrac{1}{3})$ &  \sss $(\tfrac{4}{3},-\tfrac{2}{3})$ &  \sss $(2,0)$
\\
\sss $(\tfrac13,-\tfrac13)$  &  \sss $(\tfrac23,\tfrac13)$ &  \sss $(-\tfrac{1}{3},-\tfrac{2}{3})$ &  \sss $(0,0)$
\\
[-6pt]
\sss $(0,0)$ &  \sss $(\tfrac14,\tfrac34)$ &  \sss $(-\tfrac34,-\tfrac14)$ &  \sss $(-\tfrac12,\tfrac12)$
\\
[-6pt]
\sss $(-\tfrac13,\tfrac13)$ & \sss $(-\tfrac16,\tfrac76)$ &  \sss $(-\tfrac76,\tfrac16)$  & 
 \sss $(-1,1)$
\\
 \sss $(1,-\tfrac13)$ &  \sss $(\tfrac32,\tfrac13)$ &  \sss $(\tfrac12,-\tfrac23)$ & \sss $(1,0)$
\\
[-6pt]
\sss $(\tfrac35,\tfrac15)$ & \sss $(1,1)$ &  \sss $(0,0)$ & \sss $(\tfrac25,\tfrac45)$
 \\
 [-6pt]
\sss  $(-\tfrac15,-\tfrac35)$ & \sss $(0,0)$ & \sss $(-1,-1)$ & \sss $(-\tfrac45,-\tfrac25)$
 \\
 [-6pt]
\sss $(\tfrac13,-1)$ & \sss $(\tfrac23,-\tfrac12)$ &  \sss $(-\tfrac13,-\tfrac32)$ & \sss $(0,-1)$
\end{tabular}
\bigskip
\caption{$(\ga,\de)$ for the solutions with integral Stokes data $s_1^\R, s_2^\R$.  The corresponding Stokes data and holomorphic data are listed in \cite{GuLi2XX}.}\label{integralsolutions}
\end{table}

The holomorphic data will be discussed in detail in \cite{GuLi2XX}, but we conclude this section with some brief comments. First, the smoothness of our solutions $w_i$ near $z=0$ (and their radial nature) implies that the holomorphic data must be of the special form $p_i(z)=c_i z^{k_i}$.  (Conversely, if $p_i(z)=c_i z^{k_i}$ and $k_i\ge -1$ for all $i$ then $w_i$ is smooth near $z=0$ and $w_i=w_i(\vert z\vert)$.) These $k_0,\dots,k_n$ are equivalent to $\ga,\de$ (see Table 2 of \cite{GuLiXX} or Table 2 of \cite{GuLi2XX} for explicit formulae).   Using this holomorphic data we can try to identify the geometrical objects responsible for the solutions.

The first two blocks of Table \ref{integralsolutions} correspond to quantum cohomology rings of certain complete intersection varieties.  We denote by $\X^{v_0,\dots,v_p}_{d_1,\dots,d_m}$ the variety given by the intersection of $m$ hypersurfaces of degrees $d_1,\dots,d_m$ in weighted projective space 
$\P^{v_0,\dots,v_p}$.  The varieties which arise from our solutions are shown in Table \ref{completeintersections}.  These points lie on the top and left-hand edges of the region (Fig.\ \ref{theregion}) of Theorem A. 

\begin{table}[h]
\renewcommand{\arraystretch}{1.5}
\begin{tabular}{c|c|c|c}
Cases 4a,4b & Cases 5a,5b & Cases 5c,5d,5e & Cases 6a,6b,6c
 \\
\hline
$\P^3=\P^{1,1,1,1}$ &  $\P^4=\P^{1,1,1,1,1}$ &  $\P^{1,1,1,2}$ & $\P^{1,1,1,1,2}$
\\
[0pt]
$\X^{1,1,1,6}_{2,3}$ &  $\X^{1,1,1,1,6}_{2,3}$ &    $\X^{1,1,6}_{3}$ &  $\X^{1,1,1,6}_{3}$
\\
[0pt]
$\X^{1,1,4}_{2}$ &  $\X^{1,1,1,4}_{2}$ &  $\P^{1,4}$ &  $\P^{1,1,4}$
\\
[0pt]
$\P^{1,3}$ &  $\P^{1,1,3}$ &  $\P^{2,3}$ &  $\P^{1,2,3}$
\\
[0pt]
$\P^{2,2}$ &  $\P^{1,2,2}$ &  $\P^{1,2,2}$ &  $\P^{2,2,2}$
\\
\hline
$\P^{1,3}$ &  $\P^{2,3}$ &  $\P^{1,1,3}$ &  $\P^{1,2,3}$
\\
[0pt]
$\X^{1,1,4}_{2}$ &  $\P^{1,4}$ &  $\X^{1,1,1,4}_{2}$ &  $\P^{1,1,4}$
 \\
 [0pt]
$\X^{1,1,1,6}_{2,3}$ &  $\X^{1,1,6}_{3}$ &  $\X^{1,1,1,1,6}_{2,3}$ &  $\X^{1,1,1,6}_{3}$
 \\
 [0pt]
$\P^3=\P^{1,1,1,1}$ &  $\P^{1,1,1,2}$ &  $\P^4=\P^{1,1,1,1,1}$ &  $\P^{1,1,1,1,2}$
\end{tabular}
\bigskip
\caption{Quantum cohomology interpretation for solutions with integral Stokes data.}\label{completeintersections}
\end{table}

\no The other entries of Table \ref{integralsolutions} are discussed in \cite{GuLi2XX}.  For
example (an unfolding of) the $A_4$ singularity appears as {\sss$(\tfrac35,\tfrac15)$} or 
{\sss$(-\tfrac15,-\tfrac35)$} in the first column, and the $A_5$ singularity appears as 
{\sss$(\tfrac23,\tfrac13)$} in column two or 
{\sss $(-\tfrac13,-\tfrac23)$} in column three. These are interior points of the region.

\section{Existence of solutions: proof of Theorem A}\label{pdetheorem}

In this section we prove Theorem A.  For the tt*-Toda equations, i.e.\ $a,b\in\{1,2\}$, this amounts to solving a certain Riemann-Hilbert problem, and we intend to discuss these aspects elsewhere.  Here, however, we give an elementary approach using the monotone iteration method for solving nonlinear elliptic partial differential equations, and this method works for arbitrary $a,b>0$.  Although the method is well known for scalar p.d.e., it rarely applies to systems, and it applies to our system (\ref{pde}) only through a fortuitous combination of circumstances.  

The proof of Theorem A in the case $a=b=2$ and $\ga\de\ge 0$ was given in detail in \cite{GuLiXX}.  In  section \ref{guli1} we extend that proof to the case $a=b>0$ and $\ga\de\ge 0$.  In section \ref{guli1extended} we shall extend the proof further to the case $\ga\de\le 0$. In section \ref{allab} we explain how the case $a=b>0$ implies the general case $a,b>0$. 

For notational convenience in this section\footnote{Outside this section we always use the
notation of equation (\ref{ost}); in that notation we have 
$2w_i(z) = (\ga_i+o(1)) \log \vert z\vert\  \text{as}\ \vert z\vert\to 0$.} (and ease of comparison with section 3 of \cite{GuLiXX}), we restate the equations as
\begin{equation}\label{guli1pde}
\begin{cases}
(w_0)_{z\zbar}&= \ e^{aw_0} - e^{w_1-w_0} 
\\
(w_1)_{z\zbar}&= \ e^{w_1-w_0} - e^{-bw_1}
\end{cases}
\end{equation}
for $w_0,w_1:\C^\ast\to\R$ (where $a,b>0$).
Our objective (Theorem A) is to prove that, for any $(\ga_0,\ga_1)$ in the region 
\begin{equation}\label{inequalities}
\text{$\ga_0\ge -2/a$, $\ga_1\le 2/b$, $\ga_0-\ga_1\le2$}
\end{equation}
there is a unique solution $(w_0,w_1)$ such that
\begin{gather*}
\ w_i(z) \to 0\ 
\text{as}\ \vert z\vert\to \infty
\\
\ w_i(z) = (\ga_i+o(1)) \log \vert z\vert\ 
\text{as}\ \vert z\vert\to 0.
\end{gather*}
Conditions (\ref{inequalities}) are explained in Remark 3.2 (i),(iii) of
\cite{GuLiXX}.
We emphasize that only solutions which are smooth on $\C^\ast$ are to be discussed.  Statements such as $f\le g$ mean $f(z)\le g(z)$ for all $z\in\C^\ast$.
We do not assume in advance that $w_i$ depends only on $\vert z\vert$; however, this property follows from the uniqueness of the solution.

\subsection{The case $a=b>0$, $\ga_0\ga_1\ge 0$}\label{guli1}

The case $\ga_0,\ga_1\le0$ is analogous to the case $\ga_0,\ga_1\ge0$, so
we just treat the latter.
We shall explain how to extend the proof for $a=b=2$ in \cite{GuLiXX} to the case  $a=b>0$; this will also serve as a review of the method of \cite{GuLiXX}.
The method has three main steps:

\no(a) existence of a supersolution $(0,0)$ and a subsolution $(q_0,q_1)$ (this means
showing that any solution $(w_0,w_1)$ of (\ref{guli1pde}) satisfies  $q_i\le w_i\le 0$).

\no(b) existence of a maximal solution $(w_0,w_1)$ (i.e.\ a solution $(w_0,w_1)$ exists, and
if $(\tilde w_0,\tilde w_1)$
is any other solution, then $\tilde w_i\le w_i$).

\no(c) uniqueness of the solution.

\no The proof is carried out first for the case where $\ga_0-\ga_1 < 2$, $\ga_1 < 2/b$,
then extended to $\ga_0-\ga_1 \le 2$, $\ga_1 \le 2/b$ by a limiting argument. 

\no{\em Proof of (a):}  The fact that $w_0,w_1\le0$ follows easily from the maximum principle (see Proposition 3.3 and the following page of \cite{GuLiXX}).   As in \cite{GuLiXX} we have to construct lower bounds for any solution  $(w_0,w_1)$ of (\ref{guli1pde}).    For this, let $h,q_0,q_1$ be the solutions of
\begin{align*}
&\begin{cases}
h_{z\zbar}= e^{ah} - 1\\
h(z)=(\ga_0+\ga_1+o(1))\log\vert z\vert\ \text{as}\ \vert z\vert\to 0
\end{cases}
\\
&\begin{cases}
(q_0)_{z\zbar}= e^{aq_0} - e^{h-2q_0}\\
q_0(z)=(\ga_0+o(1))\log\vert z\vert\ \text{as}\ \vert z\vert\to 0
\end{cases}
\\
&\begin{cases}
(q_1)_{z\zbar}= e^{2q_1-h} - e^{-aq_1}\\
q_1(z)=(\ga_1+o(1))\log\vert z\vert\ \text{as}\ \vert z\vert\to 0
\end{cases}
\end{align*}
on $\C^\ast$, where all solutions tend to $0$ as $\vert z\vert\to\infty$.   The existence and uniqueness of $h,q_0,q_1$ can be proved by standard p.d.e.\ methods and the maximum principle, under the assumption (\ref{inequalities}).   Furthermore we can prove that

\no(i) $h\le w_0+w_1\le 0$

\no(ii) $h \le q_0+q_1 \le  0$

\no(iii) $q_0\le w_0\le 0$ and $q_1\le w_1\le 0$

\no on $\C^\ast$.   For these proofs we refer to Lemma 3.5, Lemma 3.6 and  
Proposition 3.7 of \cite{GuLiXX}

\no{\em Proof of (b):}  
We obtain smooth functions $w_i^{(n)}$ ($n=0,1,2,\dots$) on $\C^\ast$ such that
\begin{equation}\label{decreasing}
q_i\le\cdots\le w_i^{(n+1)}\le w_i^{(n)}\le \cdots\le w_i^{(0)}\le 0
\end{equation}
in the following way.  

First, when $(\ga_0,\ga_1)$ lies in the smaller region
$[0,2)\times[0,2/a)$, we may take $w_i^{(0)}= 0$ and define
$(w_0^{(n+1)},w_1^{(n+1)})$ inductively from $(w_0^{(n)},w_1^{(n)})$
by solving the linear elliptic system 
\begin{equation*}
\begin{cases}
(w_0^{(n+1)})_{z\zbar} - (a+e^{-q_0}) w_0^{(n+1)} =
f_0(w_0^{(n)},w_1^{(n)},z)
\\
\ w_0^{(n+1)}(z)=\ga_0\log\vert z\vert+O(1)
\text{ at $0$}, 
\ w_0^{(n+1)}(z) \to 0 \text{ at $\infty$}
\end{cases}
\end{equation*}
where $f_0(s,t,z) =
e^{as} - e^{t-s} - (a+e^{-q_0})s$;
\begin{equation*}
\begin{cases}
(w_1^{(n+1)})_{z\zbar} - (e^{-q_0}+ae^{-aq_1}) w_1^{(n+1)} =
f_1(w_0^{(n)},w_1^{(n)},z)
\\
\ w_1^{(n+1)}(z)=\ga_1\log\vert z\vert+O(1)
\text{ at $0$}, 
\ w_1^{(n+1)}(z) \to 0
\text{ at $\infty$}
\end{cases}
\end{equation*}
where $f_1(s,t,z) =
e^{t-s} - e^{-at} - (e^{-q_0}+ae^{-aq_1})t$.

The subtracted terms on each side are chosen to make $\b f_0/\b s\le 0$ and $\b f_1/\b t\le 0$.
It is a key property of the tt*-Toda equations that $\b f_0/\b t\le 0$ and $\b f_0/\b s\le 0$
(see Remark 3.9 (i) of \cite{GuLiXX}).
This monotonicity allows us to establish property 
(\ref{decreasing}).  

The sequence $w_i^{(n)}$  converges to a smooth solution $w_i$ on $\C^\ast$.  To prove that it is a maximal solution, the maximum principle is used.  

To construct (maximal) solutions for  the range $\ga_0\ge0$, $0\le \ga_1<2/b$,
$\ga_0-\ga_1<2$ 
we make use of the solutions just constructed.
Let $(g_0,g_1)$ be such a solution with
$g_i(z) =\tilde\ga_i \log \vert z\vert+O(1)$ as $\vert z\vert\to 0$
and $g_i(z) \to 0$ as $\vert z\vert\to \infty$.
We choose $(\tilde\ga_0,\tilde\ga_1)
\in [0,2)\times[0,1/a)$
so that
$0\le \tilde\ga_0<\ga_0$, $0\le \tilde\ga_1<\ga_1$
and $\tilde\ga_1>\ga_0-2$.
Then we take $(w_0^{(0)},w_1^{(0)})=(g_0,g_1)$ as the first step in the iteration
given by the modified linear system
\begin{equation*}
\begin{cases}
(w_0^{(n+1)})_{z\zbar} - (a+e^{g_1-q_0}) w_0^{(n+1)} =
f_0(w_0^{(n)},w_1^{(n)},z)
\\
\ w_0^{(n+1)}(z)=\ga_0\log\vert z\vert+O(1)
\text{ at $0$}, 
\ w_0^{(n+1)}(z) \to 0 \text{ at $\infty$}
\end{cases}
\end{equation*}
where $f_0(s,t,z) =
e^{as} - e^{t-s} - (a+e^{g_1-q_0})s$;
\begin{equation*}
\begin{cases}
(w_1^{(n+1)})_{z\zbar} - (e^{g_1-q_0}+ae^{-aq_1}) w_1^{(n+1)} =
f_1(w_0^{(n)},w_1^{(n)},z)
\\
\ w_1^{(n+1)}(z)=\ga_1\log\vert z\vert+O(1)
\text{ at $0$}, 
\ w_1^{(n+1)}(z) \to 0
\text{ at $\infty$}
\end{cases}
\end{equation*}
where $f_1(s,t,z) =
e^{t-s} - e^{-at} - (e^{g_1-q_0}+ae^{-aq_1})t$.
Again the sequence $w_i^{(n)}$  converges to a maximal smooth solution $w_i$ on $\C^\ast$.  

\no{\em Proof of (c):}   

By integration we obtain 
\begin{gather*}
\smallint_{\!\R^2\,}\ -(1-e^{aw_0}) + (1-e^{w_1-w_0})
 = -2\pi \ga_0
\\
\smallint_{\!\R^2\,}\ -(1-e^{w_1-w_0}) + (1-e^{-aw_1})
= -2\pi \ga_1.
\end{gather*}
Thus 
\begin{equation}\label{sum}
\smallint_{\!\R^2\,}\ (e^{aw_0}-e^{-aw_1}) = -2\pi (\ga_0+\ga_1)
\end{equation}
holds for any solution $(w_0,w_1)$.  Suppose $(\tilde w_0, \tilde w_1)$ is a solution,
and $(w_0,w_1)$ is the maximal solution obtained in (b).  Then we have
\[
e^{a\tilde w_0}-e^{-a\tilde w_1} \ \le\  e^{aw_0}-e^{-aw_1}.
\]
Since both satisfy (\ref{sum}), we must have $\tilde w_0=w_0$ and
$\tilde w_1=w_1$.   This completes the proof of (c).

\subsection{The case $\ga_0\ga_1 < 0$}\label{guli1extended}

Let us consider the system (\ref{guli1pde}) with $a=b>0$ where $\ga_0,\ga_1$ satisfy (\ref{inequalities}) and also $\ga_0>0$, $\ga_1< 0$ (the case 
$\ga_0<0$, $\ga_1>0$ is similar).  In particular $\ga_0<2$ and $\ga_1>-2$.
We assume first that $\ga_0-\ga_1<2$;  the
case $\ga_0-\ga_1\le 2$ will follow as in section \ref{guli1}.  Thus we have
the system
\begin{equation}\label{guli1pdeab1}
\begin{cases}
(w_0)_{z\zbar}&= \ e^{aw_0} - e^{w_1-w_0} 
\\
(w_1)_{z\zbar}&= \ e^{w_1-w_0} - e^{-aw_1}
\end{cases}
\end{equation}
with boundary conditions
\[
w_i(z) = \ga_i\log \vert z\vert+O(1) \ \text{as $\vert z\vert\to 0$},
\]
and
$w_i(z) \to 0$ as $\vert z\vert\to \infty$.

We divide the proof into steps (a), (b), (c) as above.

(a) Let $\ga_1^\ast\in(-2,\ga_1)$.  The point $(0,\ga_1^\ast)$
lies in the region of applicability of section \ref{guli1}, so
we have a unique solution $(w_0^\ast,w_1^\ast)$
such that 
\begin{gather*}
\ w_0^\ast(z) = O(1),
\ w_1^\ast(z) = \ga_1^\ast\log \vert z\vert+O(1)
\ \text{as}\ \vert z\vert\to 0.
\end{gather*}
Similarly, for any $\bar\ga_0\in(\ga_0,2)$, 
there is a unique solution $(\bar w_0,\bar w_1)$ corresponding to
$(\bar\ga_0,0)$, i.e.\ such that 
\begin{gather*}
\ \bar w_0(z) = \bar\ga_0\log \vert z\vert+O(1),
\ \bar w_1(z) =  O(1)
\ \text{as}\ \vert z\vert\to 0.
\end{gather*}
Moreover, from section \ref{guli1}, we know that
$\bar w_i<0<w_i^\ast$.

We claim that any solution $(w_0,w_1)$ of (\ref{guli1pdeab1}) must satisfy
\[
\bar w_i\le w_i\le w_i^\ast.
\]
This will establish lower and upper bounds.

We shall prove first that $\bar w_0 \le w_0$.
Let us suppose that $\bar w_0(z)>w_0(z)$
for some $z$.   Then the boundary conditions
(in particular $\ga_0 < \bar\ga_0$)  imply that
\[
\bar w_0(z_0)-w_0(z_0)=\maxx \ (\bar w_0-w_0) > 0
\]
for some $z_0\in\C^\ast$.    Hence the maximum principle gives
\begin{align*}
0&\ge (\bar w_0-w_0)_{z\zbar}(z_0) \\
&=
e^{a\bar w_0(z_0)} - e^{\bar w_1(z_0)-\bar w_0(z_0)} - (e^{aw_0(z_0)} - e^{w_1(z_0)-w_0(z_0)})
\end{align*}
from which it follows that
$e^{\bar w_1(z_0)-\bar w_0(z_0)}  >  e^{w_1(z_0)-w_0(z_0)}$, which
implies that $\max\,(\bar w_1- w_1) > 0$.

Similarly, the boundary conditions (in particular $\ga_1 < 0$) imply that
\begin{equation*}
\bar w_1(z_1)-w_1(z_1) = 
\maxx \ (\bar w_1-w_1) > 0
\end{equation*}
for some $z_1\in\C^\ast$. 
Therefore 
\begin{equation}\label{1}
\bar w_1(z_1)-w_1(z_1) \ge \bar w_1(z_0)-w_1(z_0) > 
\bar w_0(z_0)-w_0(z_0) > 0.
\end{equation}

Applying the maximum principle to the second equation, we have
\begin{align*}
0&\ge (\bar w_1-w_1)_{z\zbar}(z_1)
\\
&=e^{\bar w_1(z_1)-\bar w_0(z_1)} - e^{-a\bar w_1(z_1)}
- (e^{w_1(z_1)-w_0(z_1)} - e^{-aw_1(z_1)}).
\end{align*}
Hence
\[
e^{w_1(z_1)-w_0(z_1)} \ge (e^{-aw_1(z_1)} - e^{-a\bar w_1(z_1)})+
e^{\bar w_1(z_1)-\bar w_0(z_1)} > e^{\bar w_1(z_1)-\bar w_0(z_1)}
\]
as we know $e^{-aw_1(z_1)}-e^{-a\bar w_1(z_1)}>0$.   
We obtain
\[
\bar w_0(z_0)-w_0(z_0) \ge 
\bar w_0(z_1)-w_0(z_1) >
\bar w_1(z_1)-w_1(z_1),
\]
which contradicts (\ref{1}).  This completes the proof that
$\bar w_0 \le w_0$.

The other three inequalities in the claim can be proved in a similar way.

(b) Next we apply the monotone scheme to prove that there exists a maximal solution, making use of $(w_0^\ast,w_1^\ast)$ and $(\bar w_0,\bar w_1)$ as supersolution and subsolution, respectively. 

We consider the linear equations
\begin{equation}\label{2}
\begin{cases}
(w_0^{(n+1)})_{z\zbar} - (ae^{aw_0^\ast} + e^{w_1^\ast-\bar w_0})w_0^{(n+1)} =
f_0(w_0^{(n)},w_1^{(n)},z)
\\
(w_1^{(n+1)})_{z\zbar} - (e^{w_1^\ast - \bar w_0} + ae^{-a\bar w_1})w_1^{(n+1)} =
f_1(w_0^{(n)},w_1^{(n)},z)
\end{cases}
\end{equation}
subject to the boundary conditions
\[
w_i^{(n+1)}(z)=\ga_i\log\vert z\vert + O(1)\ \text{as $\vert z\vert\to 0$}
\]
and $w_i^{(n+1)}(z) \to 0$ as $\vert z\vert\to \infty$, where
\begin{gather*}
f_0(s,t,z)=e^{as}  -  e^{t-s}  -  (ae^{aw_0^\ast(z)}+ e^{w_1^\ast(z)-\bar w_0(z)})s,
\\
f_1(s,t,z)=e^{t-s}  -  e^{-at} - (e^{w_1^\ast(z) - \bar w_0(z)} + ae^{-a\bar w_1(z)})t.
\end{gather*}
Note that $w_0^\ast$ and $\bar w_1$ are bounded at $0$ (without singularity) and
\[
e^{w_1^\ast(z)-\bar w_0(z)} = O(1) \vert z\vert^{\ga_1^\ast - \bar\ga_0}
\]
there.  We may take $\ga_1^\ast - \bar\ga_0$ close to $\ga_1-\ga_0$,  so that
$\ga_1^\ast - \bar\ga_0 > -2$.   Thus equation (\ref{2}) is solvable for
$(w_0^{(n+1)},w_1^{(n+1)})$ if $(w_0^{(n)},w_1^{(n)})$ is given.  We
obtain $(w_0^{(n)},w_1^{(n)})$ for $n=0,1,2,\dots$ by starting with 
$(w_0^{(0)},w_1^{(0)}) = (\bar w_0,\bar w_1)$. 

Note that if $(s,t)$ satisfies
\[
\bar w_0 \le s \le w_0^\ast,\ \ \bar w_1 \le t \le w_1^\ast,
\]
then
\begin{align}\label{3a}
\tfrac{\b f_0}{\b s}(s,t) &= ae^{as} - ae^{aw_0^\ast(z)} + e^{t-s} - e^{w_1^\ast(z)-\bar w_0(z)} \le 0
\\
\label{3b}
\tfrac{\b f_0}{\b t}(s,t) &= -e^{t-s} \le 0.
\end{align}
Similarly
\begin{align}\label{4a}
\tfrac{\b f_1}{\b t}(s,t) &= e^{t-s} - e^{w_1^\ast(z)-\bar w_0(z)} + ae^{-at} - 
ae^{-a\bar w_1(z)} \le 0
\\
\label{4b}
\tfrac{\b f_1}{\b s}(s,t) &= -e^{t-s} \le 0.
\end{align}
We shall use this to show that
\begin{equation}\label{induction}
\bar w_i \le w_i^{(n)} \le w_i^{(n+1)} \le w_i^\ast,\quad i=0,1
\end{equation}
for $n=0,1,2,\dots$. 
We begin with the case $n=0$.

First we shall show that $\bar w_i \le w_i^{(0)} \le w_i^{(1)}$.  Since $\bar w_i=w_i^{(0)}$,
we just have to show that $\bar w_i \le w_i^{(1)}$. 
From (\ref{2}) we have
\begin{align*}
(w_0^{(1)})_{z\zbar} -  (ae^{aw_0^\ast}+ e^{w_1^\ast-\bar w_0})w_0^{(1)}
&= e^{a\bar w_0}- e^{\bar w_1-\bar w_0} - (ae^{aw_0^\ast}+ e^{w_1^\ast-\bar w_0})\bar w_0
\\
&=(\bar w_0)_{z\zbar} - (ae^{aw_0^\ast}+ e^{w_1^\ast-\bar w_0})\bar w_0
\end{align*}
hence
\[
( w_0^{(1)} - \bar w_0)_{z\zbar} - (ae^{aw_0^\ast}+ e^{w_1^\ast-\bar w_0})(w_0^{(1)}-\bar w_0)=0,
\]
and similarly
\[
( w_1^{(1)} - \bar w_1)_{z\zbar} - (e^{w_1^\ast-\bar w_0}+ ae^{-a\bar w_1})(w_1^{(1)}-\bar w_1)=0.
\]
Since $\lim_{\vert z\vert\to 0} (w_0^{(1)} - \bar w_0)(z) = +\infty$ and
$\lim_{\vert z\vert\to \infty} (w_0^{(1)} - \bar w_0)(z) = 0$, the
maximum principle implies that $\bar w_0 \le w_0^{(1)}$. A similar proof shows that
$\bar w_1 \le w_1^{(1)}$.

Next we shall show that $w_i^{(1)} \le w_i^\ast$.
Since
$\bar w_i\le w_i^{(0)}\le w_i^\ast$, properties (\ref{3a})-(\ref{3b}) give
\begin{align*}
f_0(  w_0^{(0)},w_1^{(0)}, z) &\ge f_0(  w_0^\ast,w_1^\ast, z) 
\\
&= e^{aw_0^\ast} - e^{w_1^\ast - w_0^\ast} - (ae^{aw_0^\ast} + e^{w_1^\ast - \bar w_0})w_0^\ast
\\
&= (w_0^\ast)_{z\zbar} - (ae^{aw_0^\ast} + e^{w_1^\ast - \bar w_0})w_0^\ast.
\end{align*}
Applying this in the equation for $w_0^{(1)}$, we have
\[
(w_0^{(1)})_{z\zbar} - (ae^{aw_0^\ast} + e^{w_1^\ast - \bar w_0})w_0^{(1)} 
\ge (w_0^\ast)_{z\zbar} - (ae^{aw_0^\ast} + e^{w_1^\ast - \bar w_0})w_0^\ast,
\]
so
\[
(w_0^\ast - w_0^{(1)})_{z\zbar} - 
(ae^{aw_0^\ast} + e^{w_1^\ast - \bar w_0})(w_0^\ast - w_0^{(1)}) \le 0.
\]
Since $\lim_{\vert z\vert\to 0} (w_0^\ast - w_0^{(1)})(z) = +\infty$ and
$\lim_{\vert z\vert\to \infty} (w_0^\ast - w_0^{(1)})(z) = 0$, the
maximum principle implies that $w_0^{(1)} \le w_0^\ast$.   Similarly we
can show that  $w_1^{(1)} \le w_1^\ast$.
This completes the
proof of (\ref{induction}) for $n=0$.  

By the same argument it is not difficult to prove
(\ref{induction}) for any $n=1,2,3,\dots$.

Since $w_i^{(n)}$ is bounded and monotonically increasing with respect to $n$, 
it converges to some $w_i$ in $C^2_{\text{loc}}(\R^2\setminus\{(0,0)\})$
which evidently satisfies (\ref{guli1pdeab1}).   This completes the proof of existence.

Next we claim that the above solution $(w_0,w_1)$ is minimal in the sense that
$w_i\le \tilde w_i$ 
for any other solution $(\tilde w_0,\tilde w_1)$
with the same $(\tilde \ga_0,\tilde \ga_1) = (\ga_0,\ga_1)$.  

To prove that $w_0\le \tilde w_0$,
we use the fact
already established that
$\bar w_i\le \tilde w_i\le w_i^\ast$.
Thus
\begin{align*}
(w_0^{(n+1)})_{z\zbar} - (ae^{aw_0^\ast} + e^{w_1^\ast-\bar w_0})w_0^{(n+1)} &= f_0(w_0^{(n)},w_1^{(n)},z)
\\
&\ge f_0(\tilde w_0,\tilde w_1,z)
\\
&=
(\tilde w_0)_{z\zbar} - (ae^{aw_0^\ast} + e^{w_1^\ast-\bar w_0})\tilde w_0
\end{align*} 
if we assume $\tilde w_i \ge w_i^{(n)}$.  We shall use this to prove that 
$\tilde w_0 \ge w_0^{(n+1)}$.
We have
\[
(w_0^{(n+1)} - \tilde w_0)_{z\zbar} - 
(ae^{aw_0^\ast} + e^{w_1^\ast-\bar w_0})(w_0^{(n+1)} - \tilde w_0)\ge 0
\]
because $w_0^{(n+1)} - \tilde w_0$ is continuous up to $z=0$. Since
$w_0^{(n+1)} - \tilde w_0\to 0$ as $\vert z\vert \to\infty$, the maximum principle
gives
\[
0\ge (w_0^{(n+1)} - \tilde w_0)_{z\zbar}(z_0) \ge
(ae^{aw_0^\ast} + e^{w_1^\ast-\bar w_0})(w_0^{(n+1)} - \tilde w_0)(z_0)
>0
\]
if
$(w_0^{(n+1)} - \tilde w_0)(z_0)=\maxx(w_0^{(n+1)} - \tilde w_0)>0$.
This contradiction shows that $w_0^{(n+1)}\le \tilde w_0$. 

Since $\tilde w_0 \ge w_0^{(n)}$ holds for $n=0$,  by induction
it holds for all $n\ge 0$.  
Similarly we can prove that $\tilde w_1 \ge w_1^{(n)}$ for all $n\ge 0$.
By taking the limit $n\to\infty$ we obtain $\tilde w_i \ge w_i$ for $i=0,1$. 

(c) The proof of uniqueness of the solution of (\ref{guli1pdeab1}) is the
same as that in part (c) of section \ref{guli1}. 

It remains to extend the result for the case $\ga_0-\ga_1<2$ to  the
case $\ga_0-\ga_1\le 2$.   Let us take a sequence $(\ga_0^{(n)},\ga_1^{(n)})$
in the interior of the region
such that $(\ga_0^{(n)},\ga_1^{(n)})\to(\ga_0,\ga_1)$,
for example $\ga_1^{(n)}=\ga_1$ with $\ga_0^{(n)}<\ga_0$.  By the
maximum principle, the corresponding solution 
$(w_0^{(n)},w_1^{(n)})$ of (\ref{2}) is monotone in $n$ and satisfies
\[
u_i < w_i^{(n)} < v_i
\]
where $(u_0,u_1)$, $(v_0,v_1)$ are the solutions
of (\ref{guli1pdeab1})
corresponding to $(0,2)$ and $(-2,0)$.  In the limit $n\to\infty$, 
$(w_0^{(n)},w_1^{(n)})$ converges to a solution of of (\ref{guli1pdeab1})
corresponding to $(\ga_0,\ga_1)$. Furthermore, this solution is minimal. Uniqueness of this solution follows by integrating (\ref{guli1pdeab1}).   For these solutions, we note that the weaker asymptotic property $w_i(z) = (\ga_i+o(1)) \log \vert z\vert$ holds at $0$.

\subsection{Arbitrary $a,b>0$}\label{allab}

It suffices to discuss the case $0<a<b$ and $\ga_0,\ga_1\ge0$.   The other cases may be dealt with as in sections \ref{guli1},  \ref{guli1extended}.

Let $(\bar w_0,\bar w_1)$ be the solution to
\begin{equation*}
\begin{cases}
(\bar w_0)_{z\zbar}&= \ e^{a\bar w_0} - e^{\bar w_1-\bar w_0} 
\\
(\bar w_1)_{z\zbar}&= \ e^{\bar w_1-\bar w_0} - e^{-a\bar w_1},
\end{cases}
\end{equation*}
and 
$(\tilde w_0,\tilde w_1)$ the solution to
\begin{equation*}
\begin{cases}
(\tilde w_0)_{z\zbar}&= \ e^{b\tilde w_0} - e^{\tilde w_1-\tilde w_0} 
\\
(\tilde w_1)_{z\zbar}&= \ e^{\tilde w_1-\tilde w_0} - e^{-b\tilde w_1},
\end{cases}
\end{equation*}
with the same $\ga_0,\ga_1$. 
Note that $\ga_1\le \frac2b < \frac2a$.
These solutions exist by section \ref{guli1}.

Next, we claim that
\[
\bar w_0 \le \tilde w_0\ \text{and}\ 
\bar w_1 \le \tilde w_1.
\]
If the first inequality does not hold, then $\max\, (\bar w_0 - \tilde w_0) > 0$.  Let
\[
\bar w_0(z_0) - \tilde w_0(z_0) = \max\, (\bar w_0 - \tilde w_0) > 0.
\]
We may assume that $z_0\ne 0$.  Then the maximum principle implies that
\[
e^{a\bar w_0(z_0)} - e^{\bar w_1(z_0)-\bar w_0(z_0)} \le 
e^{b\tilde w_0(z_0)} - e^{\tilde w_1(z_0)-\tilde w_0(z_0)}.
\]
Note that $0> a \bar w_0(z_0) > a\tilde w_0(z_0) > b \tilde w_0(z_0)$,
which implies that $w_1(z_0)-\tilde w_1(z_0) > 
w_0(z_0)-\tilde w_0(z_0) >0$.   Let
\[
w_1(z_1) - \tilde w_1(z_1) = \max\, (w_1 - \tilde w_1) > 0.
\]
By the maximum principle again, we have
\[
w_0(z_1)-\tilde w_0(z_1) > w_1(z_1)-\tilde w_1(z_1) >
w_0(z_0)-\tilde w_0(z_0),
\]
which is a contradiction.   Therefore $\bar w_0 \le \tilde w_0$.  Similarly we can prove that $\bar w_1 \le \tilde w_1$.  The claim is proved.

Now, $(\bar w_0,\bar w_1)$ is a subsolution of (\ref{guli1pde}), because
$(\bar w_1)_{z\zbar} \ge  \ e^{\bar w_1-\bar w_0} - e^{-b\bar w_1}$, and
$(\tilde w_0,\tilde w_1)$ is a supersolution of (\ref{guli1pde}).
Since $(\bar w_0,\bar w_1)\le (\tilde w_0,\tilde w_1)$, the monotone scheme produces a maximal solution of (\ref{guli1pde}).   Uniqueness of the solution to
(\ref{guli1pde}) follows by integration as before.  

This completes the proof of existence and uniqueness of solutions of 
(\ref{guli1pde}) under condition (\ref{inequalities}).

\section{Stokes data: proof of Theorem B}\label{stokes}

\subsection{Monodromy data}   Given any radial solution of (\ref{ost}) and (\ref{as}) on a neighbourhood $V$ of a point $z_0$,  we have the corresponding monodromy data of
\begin{equation*}
\Psi_\mu =  \left( -\tfrac{1}{\mu^2} xW - \tfrac{1}{\mu} xw_x + xW^t\right) \Psi.
\end{equation*}
This o.d.e.\  has poles of order two at $0$ and $\infty$, so the data consists of collections of Stokes matrices $S^{(0)}_i$, $S^{(\infty)}_i$ at the poles (relating solutions on different Stokes sectors) and a connection matrix $E$ (which relates solutions near $0$ with solutions near $\infty$).   A priori this monodromy data depends on $z\in V$, but it is in fact independent of $z$; it is a \ll conserved quantity\rrr.

For convenience we introduce the new variable
\[
\zeta=\la/z=\mu/x
\]
and rewrite the o.d.e.\  as follows:
\begin{equation}\label{4ode}
\Psi_\ze =  \left( -\tfrac{1}{\ze^2} W - \tfrac{1}{\ze} xw_x + x^2W^t \right) \Psi.
\end{equation}
In this section we shall compute the Stokes data at $\zeta=\infty$, using the approach of \cite{FIKN06}, for the ten cases listed in Table \ref{tableofcases}.   We shall see that each solution referred to in Theorem A  corresponds to a single Stokes matrix at $\infty$, indeed to just two entries $s_1^\R,s_2^\R$ of this Stokes matrix.  In other words, the rest of the monodromy data is uniquely determined by $s_1^\R,s_2^\R$.  For each of the ten cases we shall derive an explicit formula for $s_1^\R,s_2^\R$ in terms of the asymptotic data $\ga,\de$.  

\subsection{Cases 4a, 4b:}

Let $\eta=1/\ze$.  Then (\ref{4ode}) has the form
\[
\Psi_\eta  = \left(  -\tfrac{1}{\eta^2} x^2W^t + O(\tfrac1\eta) \right) \Psi
\]
near $\eta=0$.  For this o.d.e.\ we shall review the formal solutions, the Stokes phenomenon 
(i.e.\ the relation between formal solutions and holomorphic solutions), 
then give the definition and computation of the Stokes matrices.   Using this we give the proof of Theorem B.  The simplicity of this calculation depends on the fact that the Stokes matrices have a very special form as a consequence of the three symmetries of the equation.

\no{\bf Step 1: Formal solutions.}

In order to apply section 1.4 of \cite{FIKN06}, we are required to diagonalize the leading term of the coefficient matrix.
To do this we observe that $W=e^{-w} \, \Pi \, e^w$ where
\[
e^w=\diag(e^{w_0},e^{w_1},e^{w_2},e^{w_3}),
\ \  
\Pi=
\bp  & 1 & & \\
 & & 1 & \\
  & & & 1\\
1   & & &
   \ep
\]
and  $\Pi= \Om \, d_4\,  \Om^{-1}$ where
\[
\Om=
\bp
1 & 1& 1 & 1\\
1 & \om & \om^2 & \om^3 \\
1 & \om^2 & \om^4 & \om^6 \\
1 & \om^3 & \om^6 & \om^9 
\ep,
\ \ 
d_4=\diag(1,\om,\om^2,\om^3)
\]
(the columns of $\Om$ are the eigenvectors of $\Pi$ with eigenvalues 
$1,\om,\om^2,\om^3$, where $\om = e^{{2\pi \i}/4}$). This gives
\[
W^t = P \, d_4\,  P^{-1},  
\ \ 
P = e^w \Om^{-1}.
\]
By Proposition 1.1 of \cite{FIKN06},  it follows that there exists a unique formal solution  of (\ref{4ode}) 
of the form
\[
\Psi_f = P\left( I + \sum_{k\ge 1} \psi_k \eta^k \right) e^{
\La_0\log\eta + \frac1\eta \frac{\La_{-1}}{-1}
}
\]
where $\La_{-1}= -x^2 d_4$.  Substitution in the o.d.e.\ shows that $\La_0=0$.

The three symmetries of $\al=(\Psi^t)^{-1}d\Psi^t$ (end of section \ref{first})  translate into the following symmetries of $f=\Psi_\ze \Psi^{-1}$:

\no{\em Cyclic symmetry: }  $d_4^{-1} f(\om\ze) d_4 = \om^{-1} f(\ze)$

\no{\em Anti-symmetry: }   $\De f(-\ze)^t \De = f(\ze)$

\no{\em Reality: }   $\De 
\overline{
f 
\left(
\tfrac{1}{x^2\bar\ze  \vphantom{ T^{T^T} }  }
\right) 
}
\De = -x^2 \ze^2 f(\ze)$  

We shall not need this \ll loop group reality\rr condition explicitly (later we shall use it implicitly, in the proof of Proposition \ref{reduction}).  However we shall need  the
following more elementary (and easily verified) property, which we take as the reality condition from now on:

\no{\em Reality: }   $\overline{  f(\bar\ze) } = f(\ze)$.

\no These lead to the following symmetries of the formal solution $\Psi_f$.
    
\begin{lemma}\label{formalsymmetries}   \ \ \ \ \ 

\no{Cyclic symmetry: }  $d_4^{-1} \,\Psi_f(\om\ze)\,\Pi=\Psi_f(\ze)$

\no{Anti-symmetry: }  $\De\, \Psi_f(-\ze)^{-t}\,(\Om\De\Om)^{-1}=\Psi_f(\ze)$

\no{Reality: }  $\overline{\Psi_f(\bar\ze)}\ \Om \,\bar\Om^{-1} = \Psi_f(\ze)$
\end{lemma}

\no Explicitly, the matrices $\Om \,\bar\Om^{-1}, \Om\De\Om$ here are
\[
 \Om \,\bar\Om^{-1} =    
 \bp
 J_1 & \\
  & J_3
\ep,
 \ \ 
 \Om\De\Om=
 \begin{cases}
 4d_4^3 \ \text{in case 4a}\\
 4d_4 \ \text{in case 4b}
 \end{cases}
\]
($J_l$ is defined in section \ref{first}).   

\begin{proof}
We give the proof of the cyclic symmetry formula; the other two are similar.  From the cyclic symmetry of $\Psi_\ze\Psi^{-1}$ we see that  $d_4^{-1}\Psi_f(\om\ze)d_4$
is also a formal solution,  hence must be $\Psi_f(\ze)A$ for some constant matrix $A$.
To find $A$ we examine 
\begin{align*}
d_4^{-1}\Psi_f(\om\ze) &= d_4^{-1} P (I + O(\tfrac1\ze) \,e^{\om\ze x^2 d_4}\\
&= P (I + O(\tfrac1\ze)) P^{-1} d_4^{-1} P  \, e^{\om\ze x^2 d_4}.
\end{align*}
Since $P^{-1} d_4^{-1} P=\Om d_4^{-1} \Om^{-1} = \Pi^{-1}$ and
$\Pi^{-1}\om d_4 \Pi=d_4$,  we obtain
\[
d_4^{-1}\Psi_f(\om\ze) = 
P (I + O(\tfrac1\ze)  \, e^{\ze x^2 d_4}\, \Pi^{-1}.
\]
Thus, $d_4^{-1}\Psi_f(\om\ze)\Pi$ is a formal solution of the same type as $\Psi_f(\ze)$,
so it must be equal to $\Psi_f(\ze)$.  This completes the proof (and shows that
$A=\Pi^{-1}d_4$).
\end{proof}

\no{\bf Step 2: Stokes phenomenon.}

The Stokes phenomenon depends on the fact that (\ref{4ode}) admits a holomorphic solution $\Psi$ on some region of the form
\[
\{ \ze\in\C\cup\infty \st \th_1<\ar\ze < \th_2, \ \vert \ze\vert > R \ \}
\]
with asymptotic expansion $\Psi_f$ whenever $\th_2-\th_1$ is sufficiently small, 
i.e.\  {\em if the sector is sufficiently narrow.}  (This will be made precise in a moment.)  On the other hand,  it is easy to see that $\Psi$ --- when it exists --- is unique {\em if the sector is sufficiently wide.}   Namely, if $\Psi,\Psi^\pr$ are two such solutions, then $\Psi^\pr=\Psi C$ for some constant $C$, and we have 
\[
C=\lim_{\ze \to \infty} \Psi^{-1}\Psi^\pr =
\lim_{\ze \to \infty} e^{-\ze x^2 d_4}
(I + O(\tfrac1\ze)) e^{\ze x^2 d_4}.
\]
The desired conclusion $C=I$ would follow if,
for any $i\ne j$,
there exists a path $\ze_t\to\infty$ in the sector such that
$
(C_{ij}=)\ 
\lim_{\ze_t \to \infty} e^{\ze_t  x^2 (\om^j - \om^i)} 
O(\tfrac1{\ze_t}) = 0.
$
Since
$\vert e^{\ze_t  x^2 (\om^j - \om^i)} \vert =
e^{\Re \, \ze_t  x^2 (\om^j - \om^i)}$,  a sufficient condition
is that the path satisfies $\Re \, \ze_t   (\om^j - \om^i) < 0$.     If
the sector contains a ray $\ar\ze = \th$ with 
$\Re \, \ze   (\om^j - \om^i) = 0$, then $\Re \, \ze   (\om^j - \om^i)$ can be made positive or negative within the sector, so we have both $C_{ij}=0$ and $C_{ji}=0$.  
Such a ray is called a Stokes ray.  Stokes rays are said to have the same type if their arguments differ by $\pi$. We conclude that $C=I$ if  the sector contains {\em at least} one Stokes ray of each type.  

It is a nontrivial fact that a solution $\Psi$ exists on the sector if the sector contains {\em at most} one Stokes ray of each type  (this is the meaning of \ll sufficiently narrow\rrr).   
A sector which contains exactly one Stokes ray of each type is called a Stokes sector.  Thus we have the fundamental principle (Theorem 1.4 of \cite{FIKN06}) that there is a unique solution $\Psi$ with asymptotic expansion $\Psi_f$ on any Stokes sector.  

\no{\bf Step 3: Definition of Stokes matrices.}

Explicitly, the condition $\Re \, \ze   (\om^j - \om^i) = 0$ means that
$\cos( \ar \ze  + \ar (\om^j - \om^i) )= 0$, i.e.\
$\ar\ze + \ar(\om^j-\om^i)\in \frac\pi 2 + \pi\Z$, $0\le i<j \le 3$, and this simply means $\ar \ze \in  \frac\pi 4\Z$.   
Thus there are 8 Stokes rays;  2 of each of 4 types. These are shown in Fig.\ \ref{stokesrays}.
\begin{figure}[h]
\begin{center}
\includegraphics[scale=0.3,trim= 0 150 0 50]{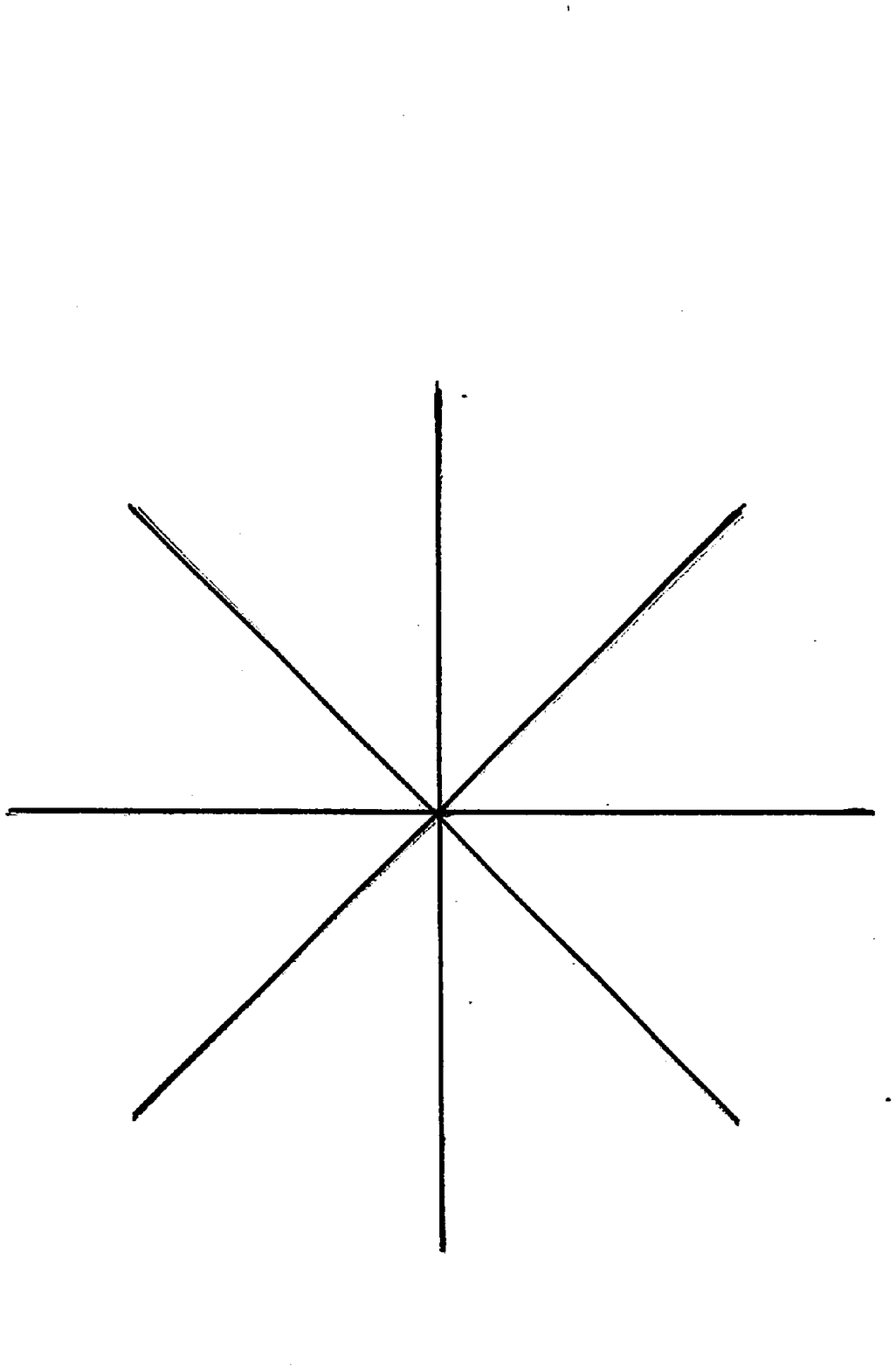}
\end{center}
\caption{Stokes rays.}\label{stokesrays}
\end{figure}
Let us choose (Fig.\ \ref{stokessector}) an initial Stokes sector
\[
\Om_1=\{ \ze\in\C \st -\tfrac\pi2<\ar\ze < \tfrac{3\pi}4\}.
\]
\begin{figure}[h]
\begin{center}
\includegraphics[scale=0.3,trim= 0 100 0 50]{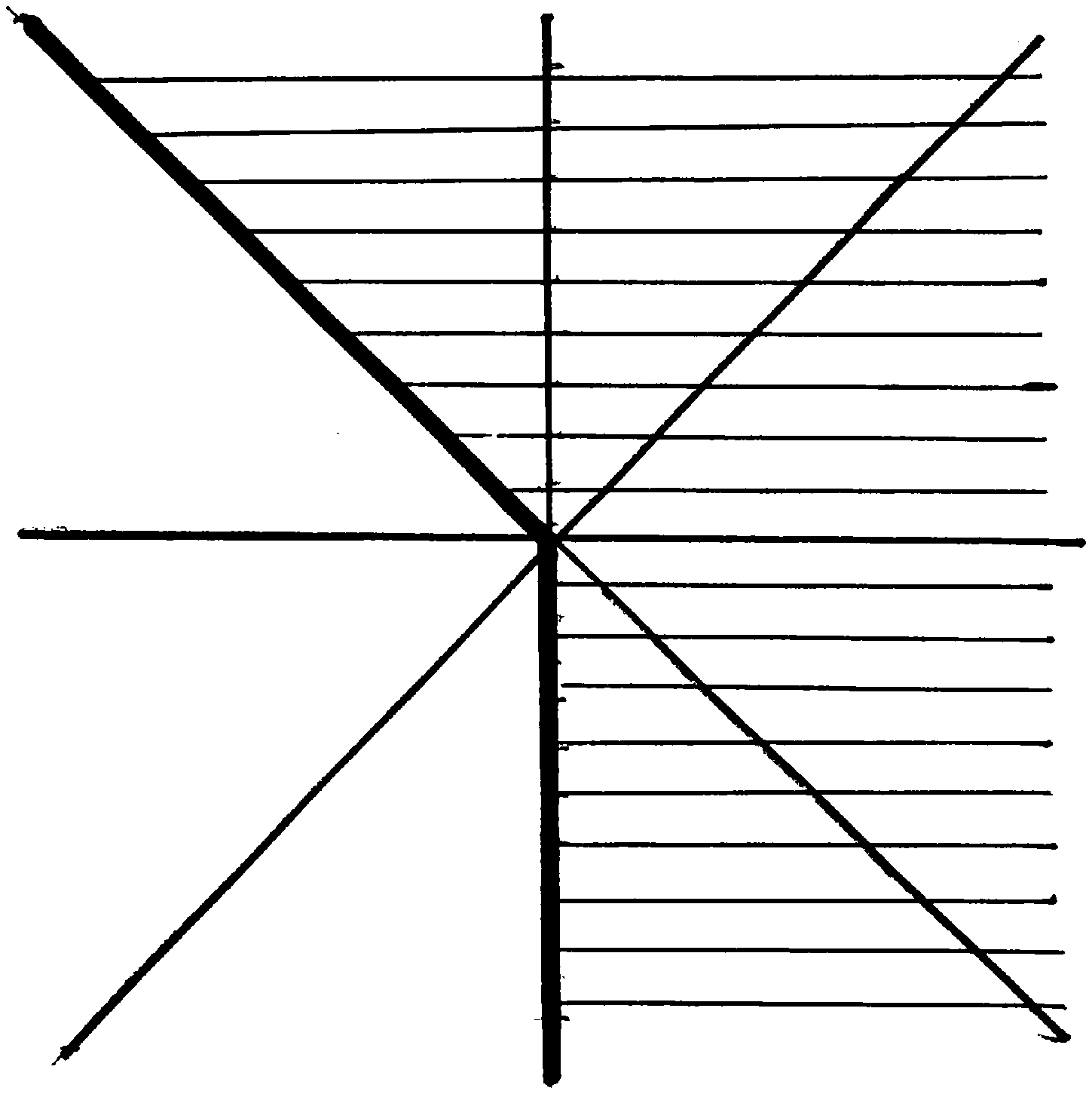}
\end{center}
\caption{Stokes sector $\Om_1$.}\label{stokessector}
\end{figure}
It will be convenient to regard $\Om_1$ as a subset of 
the universal covering surface 
$\tilde\C^\ast$ (rather than $\C^\ast$), so that we can define further subsets $\Om_k$ by
\[
\Om_{k+1}=e^{k\pi \i} \Om_1, \ \ k\in\Z.
\]
The projection of this to $\C^\ast$ is $\{ \ze\in\C^\ast \st -\tfrac\pi2+k\pi<\ar\ze < \tfrac{3\pi}4+k\pi\}$,
which is also a Stokes sector. Hence, for any $k\in\Z$,  we have a unique holomorphic solution $\Psi_k$ such that 
$\Psi_k\sim\Psi_f$ as $\ze\to\infty$ in $\Om_k$. 
Such a solution may be analytically continued to $\tilde\C^\ast$, and we shall use the same notation $\Psi_k$ for the extended solution. 

On the overlap $\Om_{k+1}\cap\Om_k$ (and hence on all of $\tilde\C^\ast$) the two solutions $\Psi_{k+1},\Psi_k$ must be related by a constant matrix. We write
\[
\Psi_{k+1}=\Psi_k S_k
\]
and call $S_k$ the $k$-th Stokes matrix (at $\infty$).   

\begin{lemma}\label{periodic}
$
S_{k+2}=S_k
$
for all $k\in\Z$.  
\end{lemma}

\begin{proof}
First, we note that the coefficients of $\Psi_f$ are holomorphic on $\C^\ast$.  
Hence $\Psi_{k+2}(e^{2\pi \i}\ze)$ and $\Psi_k(\ze)$ have the same asymptotic expansion $\Psi_f$ on $\Om_k$. It follows that they must be equal.   Applying this fact twice, we obtain
\[
\Psi_{k+2}(e^{2\pi \i}\ze)=
\begin{cases}
\Psi_{k}(\ze)=\Psi_{k-1}(\ze) S_{k-1}\\
\Psi_{k+1}(e^{2\pi \i}\ze)S_{k+1}=\Psi_{k-1}(\ze)S_{k+1}
\end{cases}
\]
which gives $S_{k+1}=S_{k-1}$, as required.   
\end{proof}

We shall in fact need a collection of Stokes sectors compatible with the symmetries.  For this, we introduce the additional sectors
\[
\Om_{k+1}=e^{k\pi \i} \Om_1, \ \ k\in\tfrac14\Z
\]
i.e.\ we allow $k\in\frac14\Z$ from now on, rather than $k\in\Z$.  Let us write
\[
\Psi_{k+\scriptstyle\frac14} = \Psi_k Q_k,
\]
so that $S_k=Q_{k}Q_{k+\scriptstyle\frac14}Q_{k+\scriptstyle\frac24}
Q_{k+\scriptstyle\frac34}$.

A similar argument to that of Lemma \ref{periodic} --- i.e.\  deducing properties of $\Psi_k$ from properties of $\Psi_f$ --- allows us to obtain the symmetries of $\Psi_k$ from those of $\Psi_f$:

\begin{lemma}\label{psik} \ \ \ \ 

\no{\em Cyclic symmetry: }  $d_4^{-1}\Psi_{k+\frac12}(\om\ze)\Pi=\Psi_k(\ze)$

\no{\em Anti-symmetry: }  $\De \Psi_{k+1}(e^{\i\pi}\ze)^{-t}(\Om\De\Om)^{-1}=\Psi_k(\ze)$

\no{\em Reality: } $\overline{  \Psi_{\frac74-k}(\bar\ze)  } \Om\bar\Om^{-1} = \Psi_k(\ze)$ 
\end{lemma}

\begin{proof}  The cyclic symmetry 
$d_4^{-1} \,\Psi_f(\om\ze)\,\Pi=\Psi_f(\ze)$ of $\Psi_f$ (Lemma \ref{formalsymmetries})
allows us to deduce that 
$d_4^{-1}\Psi_{k+\frac12}(\om\ze)\Pi=\Psi_k(\ze)$
because the left and right hand sides have the same asymptotic expansion for
$\ze\in\Om_k$ (and $\om\Om_k=\Om_{k+\frac12}$), hence must be equal. Similarly, the anti-symmetry and reality conditions follow from the corresponding properties of $\Psi_f$ and  the formulae
$e^{\i\pi}\Om_k=\Om_{k+1}$, $\bar\Om_k=\Om_{\frac74-k}$.
\end{proof}

\no{\bf Step 4: Computation of Stokes matrices.}

Each $Q_k$ (and hence each $S_k$) has a \ll triangular\rr shape which depends on our chosen diagonalization of the leading term of the o.d.e.:

\begin{lemma}\label{triangular}
(a) The diagonal entries of $Q_k$ are all $1$.  (b) For $i\ne j$, 
the $(i,j)$ entry of $Q_k$ is $0$ if $\ar(\om^i-\om^j) \not\equiv 
\frac{(7-4k)\pi}4$ mod $\pi$.
\end{lemma}

\begin{proof}  This is similar to the proof of uniqueness of $\Psi_k$.   Namely,
we have
\[
Q_k=\lim_{\ze \to \infty} \Psi_k^{-1}\Psi_{k+\scriptstyle\frac14} =
\lim_{\ze \to \infty} e^{-\ze x^2 d_4}
(I + O(\tfrac1\ze)) e^{\ze x^2 d_4}.
\]
Hence $(Q_k)_{ii}=1$ for all $i$.  Moreover, for each $(i,j)$ such that $i\ne j$,   
$(Q_k)_{ij}=0$ if there exists a path $\ze_t\to\infty$ in $\Om_k\cap\Om_{k+\scriptstyle\frac14}$ such that
$\Re \, \ze_t   (\om^j - \om^i) < 0$.    Since $\Om_k\cap\Om_{k+\scriptstyle\frac14}$
is a sector of angle $\pi$, 
$(Q_k)_{ij}$
can be nonzero  only if 
$(\om^j-\om^i)\, \Om_k\cap\Om_{k+\scriptstyle\frac14}$ is equal to the right
hand half plane $\Re\ze>0$.  This is equivalent to the condition
$\ar(\om^i-\om^j) \equiv  \frac{(7-4k)\pi}4$ mod $\pi$.
\end{proof}

Let $a_{ij}=\ar(\om^i-\om^j)$.  Then the matrix $(a_{ij})_{0\le i,j\le 3}$ is
\[
\bp
\vphantom{\frac{5^5_5}{4^4_4}}
0 &-\tfrac{\pi}{4} & 0 & \tfrac{\pi}{4}\\
\vphantom{\frac{5^5_5}{4^4_4}}
\tfrac{3\pi}{4} & 0 & \tfrac{\pi}{4} & \tfrac{2\pi}{4}\\
\vphantom{\frac{5^5_5}{4^4_4}}
\tfrac{4\pi}{4} & \tfrac{5\pi}{4} & 0 & \tfrac{3\pi}{4}\\
\vphantom{\frac{5^5_5}{4^4_4}}
\tfrac{5\pi}{4} & \tfrac{6\pi}{4} & \tfrac{7\pi}{4} & 0
\ep.
\]
It follows from this and the lemma that 
\[
Q_1=
\bp
1 & & & \\
* & 1 & & \\
 & & 1 & * \\
  & & & 1
\ep,
\ \ 
Q_{1\scriptstyle\frac14}=
\bp
1 & & & \\
 & 1 & & *\\
 & & 1 &  \\
  & & & 1
\ep.
\]
We shall only need 
$Q_1, Q_{1\scriptstyle\frac14}$, because of the following identities.

\begin{lemma}\label{qk}   We have:

\no{Cyclic symmetry: }  $Q_{k+\scriptstyle\frac12} = \Pi\,  Q_k \, \Pi^{-1}$

\no{Anti-symmetry: }  $Q_{k+1} = \Om\De\Om \, Q_k^{-t}\,  (\Om\De\Om)^{-1} =
\begin{cases}
d_4^3 \, Q_k^{-t}\, d_4^{-3}\ \text{in case 4a}\\
d_4 \, Q_k^{-t}\, d_4^{-1}\ \text{in case 4b}
\end{cases}$

\no{Reality: }  $Q_k=    
\Om \,\bar\Om^{-1}\,
\bar Q_{ {\scriptstyle\frac32} - k}^{-1}\,
(\Om \,\bar\Om^{-1})^{-1}$
\end{lemma}

\begin{proof}
These follow immediately from Lemma \ref{psik} and the definition of $Q_k$.
\end{proof}

\begin{proposition}\label{qk2}  We have
\[
Q_1=
\bp
1 & & & \\
s_1 & 1 & & \\
 & & 1 & -\bar s_1 \\
  & & & 1
\ep,
\ \ 
Q_{1\scriptstyle\frac14}=
\bp
1 & & & \\
 & 1 & & s_2\\
 & & 1 &  \\
  & & & 1
\ep
\]
where $s_1,s_2\in\C$.   Furthermore, 
\[
s_1\in
\begin{cases}
\om^{\scriptstyle\frac32}\R\ \text{in case 4a}\\
\om^{\scriptstyle\frac12}\R\ \text{in case 4b}
\end{cases}
\ \ 
s_2\in
\begin{cases}
\om^{3}\R\ \text{in case 4a}\\
\om\R\ \text{in case 4b}
\end{cases}
\]
so we can define real numbers  $s_1^{\R},s_2^{\R}$ by
\[
\text{$s_1=\om^{\scriptstyle\frac32} s_1^{\R}$, 
$s_2=\om^{3} s_2^{\R}$
in case 4a,}
\]
and
\[
\text{$s_1=\om^{\scriptstyle\frac12} s_1^{\R}$, 
$s_2=\om s_2^{\R}$
in case 4b.}
\]
\end{proposition}

\begin{proof}  We know that
\[
Q_1=
\bp
1 & & & \\
q_{10} & 1 & & \\
 & & 1 & q_{23} \\
  & & & 1
\ep,
\ \ 
Q_{1\scriptstyle\frac14}=
\bp
1 & & & \\
 & 1 & & q_{13}\\
 & & 1 &  \\
  & & & 1
\ep
\]
for some $q_{10}, q_{23}, q_{13} \in\C$.   The reality condition (at $\infty$) gives $q_{10}+\bar q_{23}=0$, $q_{13}+\bar q_{13}=0$.  
The anti-symmetry condition gives 
$q_{23}+\om q_{10}=0$ in case 4a,  and
$q_{10}+\om q_{23}=0$ in case 4b.
The stated  result follows.
\end{proof}

\begin{corollary}\label{charac}
The monodromy at $\infty$ is  
$(Q_1 Q_{1\scriptstyle\frac14}\Pi)^4$.
The characteristic polynomial of $Q_1 Q_{1\scriptstyle\frac14}\Pi$
is
\[
\la^4 - s_1 \la^3 - s_2 \la^2 + \bar s_1 \la - 1.  
\]
If $\la$ is a root of this polynomial,
then so is $\frac{1}{\om\la}$ in case 4a, and $\frac{1}{\om^3\la}$ in case 4b.
\end{corollary}

\begin{proof}
The monodromy is (the conjugacy class of) $S_1S_2$.  From the cyclic symmetry we have
$S_1=Q_{1}Q_{1\scriptstyle\frac14}Q_{1\scriptstyle\frac24}
Q_{1\scriptstyle\frac34}=
Q_{1}Q_{1\scriptstyle\frac14}\Pi Q_{1}Q_{1\scriptstyle\frac14}\Pi^{-1}$, and
$S_2=Q_{2}Q_{2\scriptstyle\frac14}Q_{2\scriptstyle\frac24}
Q_{2\scriptstyle\frac34}=
\Pi^2 Q_{1}Q_{1\scriptstyle\frac14} \Pi^{-2} \,
\Pi^3 Q_{1}Q_{1\scriptstyle\frac14} \Pi^{-3}$. 
Since $\Pi^{-3}=\Pi$, we obtain $S_1S_2=(Q_1 Q_{1\scriptstyle\frac14}\Pi)^4$.   The characteristic polynomial of $Q_1 Q_{1\scriptstyle\frac14}\Pi$ may be computed from the formulae in the proposition.  The final statement follows from the anti-symmetry condition.
\end{proof} 

Let us note at this point that our calculations have the following consequence:

\begin{proposition}\label{reduction} The monodromy data $S^{(0)}_i, S^{(\infty)}_i, E$ ($i\in\Z$)
for the linear system (\ref{4ode}) is equivalent
to $s_1^\R,s_2^\R$.
\end{proposition} 

\begin{proof}  We have shown in Lemma \ref{periodic} that the Stokes matrices 
$S_k=S^{(\infty)}_k$ reduce to $S_1,S_2$.   By Lemma \ref{qk}, they reduce to $S_1$.  By Proposition \ref{qk2} they reduce to
$s_1^\R,s_2^\R$.

By combining the reality and \ll loop group reality\rr conditions, the $S^{(0)}_k$ can be expressed in terms of the $S^{(\infty)}_k$.   

Concerning the connection matrix $E$,  we can argue indirectly that this is also determined by $s_1^\R,s_2^\R$, because  Proposition \ref{injective} 
(below) and Theorem A say that $s_1^\R,s_2^\R$ determine $\ga,\de$ and hence the functions $u,v$.  (We compute $E$ in \cite{GuItLi2XX}.)
\end{proof}

It remains to establish the formulae for $s_1^\R,s_2^\R$ announced in 
section \ref{results}:

\no{\bf Proof of Theorem B(i):}    For any $(\ga,\de)$ in the region
\[
\text{$\ga\ge -1$, $\de\le 1$, $\ga-\de\le 2$,}
\]
Theorem A gives a solution $u,v$ of the system
\[
u_{z \zbar} = e^{2u} - e^{v-u},\ \  v_{z \zbar} =  e^{v-u}  - e^{-2v}.
\]
This corresponds to a solution $w_0,w_1,w_2,w_3$
of the system (\ref{ost}) in cases 4a, 4b of Table \ref{tableofcases}.
Our aim is to prove, in both cases, that
\begin{align*}
\pm s_1^\R &= 2\cos \tfrac\pi4 {\scriptstyle (\ga+1)} \ +\  2\cos \tfrac\pi4 
{\scriptstyle (\de+3)}\\
-s_2^\R &= 2+4\cos \tfrac\pi4 {\scriptstyle (\ga+1)} \ \cos \tfrac\pi4 
{\scriptstyle (\de+3)}.
\end{align*}
We shall do this by computing the (conjugacy class of the) monodromy at $\infty$ in two ways. First, by Corollary \ref{charac}, the monodromy is 
$(Q_1 Q_{1\scriptstyle\frac14}\Pi)^4$. On the other hand, the monodromy can be computed from the o.d.e.\ 
\[
\mu\Psi_\mu =  \left( -\tfrac{1}{\mu} xW - xw_x + x\mu W^t \right)\Psi,
\]
by using the known asymptotic behaviour of $u,v$.  
Namely, by letting $x\to 0$, we obtain the constant coefficient o.d.e.\ 
\[
\mu\Psi_\mu = -\diag(\tfrac12\ga_0,\tfrac12\ga_1,\tfrac12\ga_2,\tfrac12\ga_3)\Psi,
\]
where $(\ga_0,\ga_1,\ga_2,\ga_3)=(\ga,\de,-\de,-\ga)$ in case 4a, and
 $(\ga_0,\ga_1,\ga_2,\ga_3)=(\de,-\de,-\ga,\ga)$ in case 4b.
We deduce that the eigenvalues of $(Q_1 Q_{1\scriptstyle\frac14}\Pi)^4$
are $e^{\scriptstyle \pi\i \ga_0}$,
$e^{\scriptstyle \pi\i \ga_1}$,
$e^{\scriptstyle \pi\i \ga_2}$,
$e^{\scriptstyle \pi\i \ga_3}$.
Hence there exist integers $a,b,c,d$ such that
the eigenvalues of $Q_1 Q_{1\scriptstyle\frac14}\Pi$
are 
\[
\om^a e^{\scriptstyle \frac\pi4\i \ga_0},
\om^b e^{\scriptstyle \frac\pi4\i \ga_1},
\om^c e^{\scriptstyle \frac\pi4\i \ga_2},
\om^d e^{\scriptstyle \frac\pi4\i \ga_3}.
\]
Now, when $\ga_i=0$ for all $i$, we must have $w_i=0$ for all $i$, hence $Q_k=I$ for all $k$, and the eigenvalues of $Q_1 Q_{1\scriptstyle\frac14}\Pi$
are $1,\om,\om^2,\om^3$.  Thus $(a,b,c,d)$ must be a permutation of $(0,1,2,3)$.    To compute $a,b,c,d$ we consider cases 4a, 4b separately.

\no{\em Case 4a:}      

By the last statement of Corollary \ref{charac}, 
$
(a,b,c,d)=(l,m,-m-1,-l-1)
$
for some $l,m\in \Z/4\Z$.   
Using Proposition \ref{qk2}, and the fact that
$s_1,-s_2$ are the first and second elementary symmetric functions of the roots of the characteristic polynomial, we obtain
\begin{align*}
s_1^\R
&=\om^{\scriptstyle-\frac32}  s_1=
e^{  {\scriptscriptstyle -\frac{3\pi \i}{\four}}} s_1\\
&=
e^{  {\scriptscriptstyle \frac{\pi \i}{\four}}(\ga+2l-3)}+
e^{  {\scriptscriptstyle \frac{\pi \i}{\four}}(\de+2m-3)}+
e^{  {\scriptscriptstyle \frac{\pi \i}{\four}}(-\de-2m-5)}+
e^{  {\scriptscriptstyle \frac{\pi \i}{\four}}(-\ga-2l-5)}\\
&=
2\cos{  \tfrac{\pi }{4}  {\scriptstyle (\ga+2l-3)}  }+
2\cos{ \tfrac{\pi }{4}  {\scriptstyle (\de+2m-3)}  }
\end{align*}
and similarly
\begin{align*}
-s_2^\R
&=-\om^{-3}s_2 =-e^{  {\scriptscriptstyle -\frac{6\pi \i}{\four}}} s_2\\
&=2+2\cos{  \tfrac{\pi }{4}  {\scriptstyle (\ga+\de+2l+2m-6)}  }+
2\cos{  \tfrac{\pi }{4} {\scriptstyle (\ga-\de+2l-2m)}  }\\
&=
2+4\cos{ {\tfrac{\pi }{4}}  {\scriptstyle(\ga+2l-3)}  }
\,\cos{ {\tfrac{\pi }{4}}  {\scriptstyle(\de+2m-3)}  }.
\end{align*}
Of the eight values of $(l,m)$ such that
$\{l,m,-m-1,-l-1\}=\{0,1,2,3\}$, 
we shall show that only $(l,m)=(0,1)$ or $(2,3)$ are possible. To eliminate the other six cases $(l,m)=(1,0), (3,2), (0,2), (2,0), (3,1), (1,3)$ we use the following fact:

\begin{proposition}\label{injective}
The map 
$(\ga,\de)\mapsto (s_1^\R,s_2^\R)$ is injective on the region $\ga\ge -1$, $\de\le 1$, $\ga-\de\le 2$.
\end{proposition}

\begin{proof} Theorem A implies that the map $(\ga,\de)\mapsto(u,v)$ is bijective.  Moreover, the proof shows that the dependence of $u$ on $\ga$ (and that of $v$ on $\de$) is monotone, hence
the map $(x,\ga,\de)\mapsto(x,u(x),v(x))$ is bijective as well.   
On the other hand, Proposition 2.2 of \cite{FIKN06} implies that the map
$
-\tfrac{1}{\mu^2} xW - \tfrac{1}{\mu} xw_x + xW^t
\mapsto (x,s_1^\R,s_2^\R)
$  
is injective;  the failure of the Riemann-Hilbert map
$
-\tfrac{1}{\mu^2} xW - \tfrac{1}{\mu} xw_x + xW^t
\mapsto (s_1^\R,s_2^\R)
$ 
to be injective is entirely due to the isomonodromic deformation parameter $x$.
We conclude that 
$(x,\ga,\de)\mapsto (x,u(x),v(x))
\mapsto (x,s_1^\R,s_2^\R)$  is  injective, hence so is
$(\ga,\de)\mapsto (s_1^\R,s_2^\R)$.
\end{proof} 

To examine injectivity of our eight candidates, it suffices to consider the expression
\[
\left(
\cos \tfrac{\pi }{4}  {\scriptstyle(\ga+2l-3)}  +
\cos \tfrac{\pi }{4}  {\scriptstyle(\de+2m-3)}  ,
\cos \tfrac{\pi }{4}  {\scriptstyle(\ga+2l-3)}
\cos \tfrac{\pi }{4}  {\scriptstyle(\de+2m-3)} 
\right).
\]
Evidently the map $(x ,y)\mapsto (\cos f(x) +\cos g(y),\cos f(x) \cos g(y))$ 
fails to be injective on the region $y\ge x $ in $[0,\pi]\times [0,\pi]$
if $x \mapsto \cos f(x) $ fails to be injective on $[0,\pi]$, or 
if $y\mapsto \cos g(y)$ fails to be injective on $[0,\pi]$.   Since $z\mapsto \cos z$ is injective only on the intervals $[n\pi,(n+1)\pi]$,
this criterion eliminates $l=1,3$ and $m=0,2$,  hence all six cases listed above.

On the other hand it is easy to verify that the map is injective in the cases
$(l,m)=(0,1)$ or $(2,3)$.   These give the stated values for $s_1^\R,s_2^\R$.

\no{\em Case 4b:} 

This time Corollary \ref{charac} shows that
$
(a,b,c,d)=(m,-m+1,-l+1,l)
$
for some $l,m\in \Z/4\Z$.   We obtain
\begin{align*}
s_1^\R
&=\om^{\scriptstyle-\frac12}  s_1=
e^{  {\scriptscriptstyle -\frac{\pi \i}{\four}}} s_1\\
&=
e^{  {\scriptscriptstyle \frac{\pi \i}{\four}}(\de+2m-1)}+
e^{  {\scriptscriptstyle \frac{\pi \i}{\four}}(-\de-2m+1)}+
e^{  {\scriptscriptstyle \frac{\pi \i}{\four}}(-\ga-2l+1)}+
e^{  {\scriptscriptstyle \frac{\pi \i}{\four}}(\ga+2l-1)}\\
&=
2\cos{  \tfrac{\pi }{4}  {\scriptstyle (\ga+2l-1)}  }+
2\cos{ \tfrac{\pi }{4}  {\scriptstyle (\de+2m-1)}  }
\end{align*}
and 
\begin{align*}
-s_2^\R
&=-\om^{-1}s_2 =e^{  {\scriptscriptstyle -\frac{2\pi \i}{\four}}} s_2\\
&=2+2\cos{  \tfrac{\pi }{4}  {\scriptstyle (\ga+\de+2l+2m-2)}  }+
2\cos{  \tfrac{\pi }{4} {\scriptstyle (\ga-\de+2l-2m)}  }\\
&=
2+4\cos{ {\tfrac{\pi }{4}}  {\scriptstyle(\ga+2l-1)}  }
\,\cos{ {\tfrac{\pi }{4}}  {\scriptstyle(\de+2m-1)}  }.
\end{align*}
This time $l=0,2$ and $m=1,3$ are eliminated as they would render the map
$(\ga,\de)\mapsto(s_1^\R,s_2^\R)$
non-injective.  We conclude that $(l,m)=(1,2)$ or $(3,0)$.  Again, we obtain the stated values for $s_1^\R,s_2^\R$. 

\subsection{Cases 5a, 5b:}  We shall just indicate the differences from the above argument.  The formal solution 
has 
$\La_{-1} = P^{-1} (-x^2W^t)P = -x^2d_5$,  with 
$P = e^w \Om^{-1}$.   Here 
$e^w=\diag(e^{w_0},e^{w_1},e^{w_2},e^{w_3},e^{w_4}),
\Om=(\om^{ij})_{0\le i,j\le 4}$, with $\om = e^{{2\pi \i}/5}$.  We shall
also use $\Pi=\Om d_5 \Om^{-1}$.

Lemma \ref{formalsymmetries}
holds (with $d_5$ instead of $d_4$), but this time we have
\[
 \Om \,\bar\Om^{-1} =
\bp
 J_1 & \\
  & J_4
\ep,
 \ \ 
 \Om\De\Om=
 \begin{cases}
 5d_5^4 \ \text{in case 5a}\\
 5d_5^2 \ \text{in case 5b}
 \end{cases}
\]
The matrix $(  \ar(\om^i-\om^j)  )_{0\le i,j\le 4}$ is
\[
\bp
\vphantom{\frac{5^5_5}{4^4_4}}
0 &\tfrac{-3\pi}{10} & \tfrac{-\pi}{10} & \tfrac{\pi}{10}& \tfrac{3\pi}{10}\\
\vphantom{\frac{5^5_5}{4^4_4}}
\tfrac{7\pi}{10} & 0 & \tfrac{\pi}{10} & \tfrac{3\pi}{10} & \tfrac{5\pi}{10}\\
\vphantom{\frac{5^5_5}{4^4_4}}
\tfrac{9\pi}{10} & \tfrac{11\pi}{10} & 0 & \tfrac{5\pi}{10} & \tfrac{7\pi}{10}\\
\vphantom{\frac{5^5_5}{4^4_4}}
\tfrac{11\pi}{10} & \tfrac{13\pi}{10} & \tfrac{15\pi}{10} & 0 & \tfrac{9\pi}{10}\\
\vphantom{\frac{5^5_5}{4^4_4}}
\tfrac{13\pi}{10} & \tfrac{15\pi}{10} & \tfrac{17\pi}{10} & \tfrac{19\pi}{10} & 0
\ep.
\]
The Stokes rays are given by $\ar\ze \in \tfrac\pi5\Z$. 
We choose
\[
\Om_1=\{ \ze\in\C \st -\tfrac{3\pi}{5}<\ar\ze < \tfrac{3\pi}{5}\}
\]
as initial Stokes sector, then
$\Om_{k+1}=e^{k\pi \i} \Om_1$, $k\in\tfrac15\Z$.

Versions of Lemmas \ref{psik} and \ref{qk} hold here with the following modifications:

$d_5^{-1}\Psi_{k+\frac25}(\om\ze)\Pi=\Psi_k(\ze)$ (cyclic) and
$\overline{  \Psi_{2-k}(\bar\ze)  } \Om\bar\Om^{-1} = \Psi_k(\ze)$ (reality) in
Lemma \ref{psik};  

$Q_{k+\frac25}=\Pi Q_k \Pi^{-1}$ (cyclic) and 
$Q_k=    
\Om \,\bar\Om^{-1}\,
\bar Q_{ {\scriptstyle\frac95} - k}^{-1}\,
(\Om \,\bar\Om^{-1})^{-1}$ (reality)  in
Lemma \ref{qk}.

The analogue of Lemma \ref{triangular} is that, for $i\ne j$, 
the $(i,j)$ entry of $Q_k$ is $0$ if $\ar(\om^i-\om^j) \not\equiv 
\frac{(19-10k)\pi}{10}$ mod $2\pi$.    Thus the fundamental matrices 
$Q_{1}, Q_{1\scriptstyle\frac15}$ have the form
$Q_{1}=I+q_{20}E_{20}+q_{34}E_{34}$,
$Q_{1\scriptstyle\frac15}=I+q_{10}E_{10}+q_{24}E_{24}$
where $E_{ij}$ denotes the matrix whose only nonzero entry is $1$ in
the $(i,j)$ position ($0\le i,j\le 4$).
The reality condition gives $q_{20}+\bar q_{24}=0$, $q_{34}+\bar q_{10}=0$. Let
us write
\[
s_1=q_{10},\ \ s_2=q_{20}.
\]
Then the anti-symmetry conditions imply that
\[
s_1=
\begin{cases}
\om^{2} s_1^{\R}\ \text{in case 5a}\\
\om\, s_1^{\R}\ \text{in case 5b}
\end{cases}
\ \ 
s_2=
\begin{cases}
\om^{4} s_2^{\R}\ \text{in case 5a}\\
\om^2 s_2^{\R}\ \text{in case 5b}
\end{cases}
\]
where $s_1^{\R},s_2^{\R}$ are real.  

Finally, in analogy to Corollary \ref{charac}, 
the monodromy at $\infty$ is  
$(Q_1 Q_{1\scriptstyle\frac15}\Pi)^5$, and the
characteristic polynomial of $Q_1 Q_{1\scriptstyle\frac15}\Pi$
is
\[
\la^5  - s_1\la^4 - s_2 \la^3 + \bar s_2 \la^2 + \bar s_1 \la - 1.  
\]
If $\la$ is a root of this polynomial,
then so is $\frac{1}{\om\la}$ in case 5a, and $\frac{1}{\om^3\la}$ in case 5b.

\no{\bf Proof of Theorem B(ii):}  
To compute $s_1^\R, s_2^\R$ in cases 5a, 5b we use the fact that
the eigenvalues of
$(Q_1 Q_{1\scriptstyle\frac15}\Pi)^5$
are
$e^{\scriptstyle \pi\i \ga_0}$,
$e^{\scriptstyle \pi\i \ga_1}$,
$e^{\scriptstyle \pi\i \ga_2}$,
$e^{\scriptstyle \pi\i \ga_3}$,
$e^{\scriptstyle \pi\i \ga_4}$.
Hence there exist integers $a,b,c,d,e$ such that
the eigenvalues of $Q_1 Q_{1\scriptstyle\frac15}\Pi$
are 
\[
\om^a e^{\scriptstyle \frac\pi5\i \ga_0},
\om^b e^{\scriptstyle \frac\pi5\i \ga_1},
\om^c e^{\scriptstyle \frac\pi5\i \ga_2},
\om^d e^{\scriptstyle \frac\pi5\i \ga_3},
\om^e e^{\scriptstyle \frac\pi5\i \ga_4}.
\]
To determine $a,b,c,d,e$ we consider each case separately.

\no{\em Case 5a:}      

From the symmetry $\la\mapsto \frac{1}{\om\la}$,  we have
$
(a,b,c,d,e)=(l,m,2,-m-1,-l-1)
$
for some $l,m\in \Z/5\Z$.   
We obtain
\begin{align*}
s_1^\R
&=\om^{-2}  s_1=
e^{  {\scriptscriptstyle -\frac{4\pi \i}{\five}}} s_1\\
&=
e^{  {\scriptscriptstyle \frac{\pi \i}{\five}}(\ga+2l-4)}+
e^{  {\scriptscriptstyle \frac{\pi \i}{\five}}(\de+2m-4)}+1+
e^{  {\scriptscriptstyle \frac{\pi \i}{\five}}(-\de-2m-6)}+
e^{  {\scriptscriptstyle \frac{\pi \i}{\five}}(-\ga-2l-6)}\\
&=
1+2\cos{  \tfrac{\pi }{5}  {\scriptstyle (\ga+2l-4)}  }+
2\cos{ \tfrac{\pi }{5}  {\scriptstyle (\de+2m-4)}  }
\end{align*}
and similarly   
\begin{align*}
-s_2^\R
&=-\om^{-4}s_2 =-e^{  {\scriptscriptstyle -\frac{8\pi \i}{\five}}} s_2\\
&=2+2\cos{  \tfrac{\pi }{5}  {\scriptstyle (\ga+2l-4)}  }+
2\cos{ \tfrac{\pi }{5}  {\scriptstyle (\de+2m-4)}  }\\
&\ \ \ \ \ \ \ \ \ \ \ \ \ \ \ \ \ \ \ \ 
+
2\cos{  \tfrac{\pi }{5}  {\scriptstyle (\ga+\de+2l+2m-8)}  }+
2\cos{  \tfrac{\pi }{5} {\scriptstyle (\ga-\de+2l-2m)}  }\\
&=2+2\cos{  \tfrac{\pi }{5}  {\scriptstyle (\ga+2l-4)}  }+
2\cos{ \tfrac{\pi }{5}  {\scriptstyle (\de+2m-4)}  }+
4\cos{  \tfrac{\pi }{5}  {\scriptstyle (\ga+2l-4)}  }
\cos{ \tfrac{\pi }{5}  {\scriptstyle (\de+2m-4)}  }.
\end{align*}
The analogue of Proposition \ref{injective} is that the map 
$(\ga,\de)\mapsto (s_1^\R,s_2^\R)$ is injective on the region 
$\ga\ge -1$, $\de\le 2$, $\ga-\de\le 2$.  
To investigate injectivity, it suffices to investigate
\[
\left(
\cos \tfrac{\pi }{5}  {\scriptstyle(\ga+2l-4)}  +
\cos \tfrac{\pi }{5}  {\scriptstyle(\de+2m-4)}  ,
\cos \tfrac{\pi }{5}  {\scriptstyle(\ga+2l-4)}
\cos \tfrac{\pi }{5}  {\scriptstyle(\de+2m-4)} 
\right).
\]
Now, the map $\ga\mapsto \cos \tfrac{\pi }{5}  {\scriptstyle(\ga+2l-4)} $
is injective on $[-1,4]$ only for $l\equiv 0 \mod 5$. The map $\de\mapsto \cos \tfrac{\pi }{5}  {\scriptstyle(\ga+2m-4)} $
is injective on $[-3,2]$ only for $m\equiv 1 \mod 5$.  We conclude that
$
(l,m)=(0,1).
$
This gives the stated values for $s_1^\R,s_2^\R$.

\no{\em Case 5b:}      

Using the symmetry $\la\mapsto \frac{1}{\om^3\la}$,  we have
$
(a,b,c,d,e)=(m,1,-m-3,-l-3,l).
$
Hence
\begin{align*}
s_1^\R
&=\om^{-1}  s_1=
e^{  {\scriptscriptstyle -\frac{2\pi \i}{\five}}} s_1\\
&=
1+2\cos{  \tfrac{\pi }{5}  {\scriptstyle (\ga+2l-2)}  }+
2\cos{ \tfrac{\pi }{5}  {\scriptstyle (\de+2m-2)}  }
\end{align*}
and
\begin{align*}
-s_2^\R
&=-\om^{-2}s_2 =-e^{  {\scriptscriptstyle -\frac{4\pi \i}{\five}}} s_2\\
&=2+2\cos{  \tfrac{\pi }{5}  {\scriptstyle (\ga+2l-2)}  }+
2\cos{ \tfrac{\pi }{5}  {\scriptstyle (\de+2m-2)}  }+
4\cos{  \tfrac{\pi }{5}  {\scriptstyle (\ga+2l-2)}  }
\cos{ \tfrac{\pi }{5}  {\scriptstyle (\de+2m-2)}  }.
\end{align*}
The injectivity criterion forces $l=4$, $m=0$.  This gives the same values for $s_1^\R,s_2^\R$ as in case 5a.

\subsection{Cases 5c, 5d, 5e:}  

The matrices $Q_k$ and their cyclic/reality/anti-symmetry properties
are the same as in cases 5a, 5b, except that in the anti-symmetry formula (Lemma \ref{qk}) we have
\[
\Om\De\Om=
 \begin{cases}
 5d_5^3 \ \text{in case 5c}\\
 5I\  \  \text{in case 5d}\\
 5d_5 \ \text{in case 5e}
 \end{cases}
\]
This leads to
\[
s_1=
\begin{cases}
\om^{4} s_1^{\R}\ \text{in case 5c}\\
\ \ \  s_1^{\R}\ \text{in case 5d}\\
\om^{3} s_1^{\R}\ \text{in case 5e}
\end{cases}
\ \ 
s_2=
\begin{cases}
\om^{3} s_2^{\R}\ \text{in case 5c}\\
\ \  \ s_2^{\R}\ \text{in case 5d}\\
\om\, s_2^{\R}\ \,\text{in case 5e}
\end{cases}
\]
with $s_1^{\R},s_2^{\R}$  real.   
If $\la$ is a root of the characteristic polynomial of $Q_1 Q_{1\scriptstyle\frac15}\Pi$,
then so is $\frac{1}{\om^2\la}$ in case 5c,  $\frac{1}{\la}$ in case 5d, and
$\frac{1}{\om^4\la}$ in case 5e.

\no{\bf Proof of Theorem B(iii):}   
To compute $s_1^\R, s_2^\R$ in cases 5c, 5d, 5e we just have to identify $(a,b,c,d,e)$. 

\no{\em Case 5c:}      

Here 
$
(a,b,c,d,e)=(l,m,-m-2,-l-2,4)
$
for some $l,m\in \Z/5\Z$.   
This gives
\begin{align*}
s_1^\R
&=\om^{-4}  s_1=
e^{  {\scriptscriptstyle -\frac{8\pi \i}{\five}}} s_1\\
&=
1+2\cos{  \tfrac{\pi }{5}  {\scriptstyle (\ga+2l-8)}  }+
2\cos{ \tfrac{\pi }{5}  {\scriptstyle (\de+2m-8)}  }
\end{align*}
and 
\begin{align*}
-s_2^\R
&=-\om^{-3}s_2 =-e^{  {\scriptscriptstyle -\frac{6\pi \i}{\five}}} s_2\\
&=2+2\cos{  \tfrac{\pi }{5}  {\scriptstyle (\ga+2l-8)}  }+
2\cos{ \tfrac{\pi }{5}  {\scriptstyle (\de+2m-8)}  }+
4\cos{  \tfrac{\pi }{5}  {\scriptstyle (\ga+2l-8)}  }
\cos{ \tfrac{\pi }{5}  {\scriptstyle (\de+2m-8)}  }.
\end{align*}
The injectivity criterion forces $l=0$, $m=1$.  This gives the stated values for $s_1^\R,s_2^\R$.

\no{\em Case 5d:}    

Here 
$
(a,b,c,d,e)=(0,l,m,-m,-l)
$
for some $l,m\in \Z/5\Z$.   
This gives
\[
s_1^\R
= s_1
=
1+2\cos{  \tfrac{\pi }{5}  {\scriptstyle (\ga+2l)}  }+
2\cos{ \tfrac{\pi }{5}  {\scriptstyle (\de+2m)}  }
\]
and 
\[
-s_2^\R
=-s_2 
=2+2\cos{  \tfrac{\pi }{5}  {\scriptstyle (\ga+2l)}  }+
2\cos{ \tfrac{\pi }{5}  {\scriptstyle (\de+2m)}  }+
4\cos{  \tfrac{\pi }{5}  {\scriptstyle (\ga+2l)}  }
\cos{ \tfrac{\pi }{5}  {\scriptstyle (\de+2m)}  }.
\]
The injectivity criterion forces $l=1$, $m=2$.  This gives the same values for $s_1^\R,s_2^\R$
as in case 5c.

\no{\em Case 5e:}      

Here 
$
(a,b,c,d,e)=(m,-m-4,-l-4,3,l)
$
for some $l,m\in \Z/5\Z$.   
This gives
\begin{align*}
s_1^\R
&=\om^{-3}  s_1=
e^{  {\scriptscriptstyle -\frac{6\pi \i}{\five}}} s_1\\
&=
1+2\cos{  \tfrac{\pi }{5}  {\scriptstyle (\ga+2l-6)}  }+
2\cos{ \tfrac{\pi }{5}  {\scriptstyle (\de+2m-6)}  }
\end{align*}
and 
\begin{align*}
-s_2^\R
&=-\om^{-1}s_2 =-e^{  {\scriptscriptstyle -\frac{2\pi \i}{\five}}} s_2\\
&=2+2\cos{  \tfrac{\pi }{5}  {\scriptstyle (\ga+2l-6)}  }+
2\cos{ \tfrac{\pi }{5}  {\scriptstyle (\de+2m-6)}  }+
4\cos{  \tfrac{\pi }{5}  {\scriptstyle (\ga+2l-6)}  }
\cos{ \tfrac{\pi }{5}  {\scriptstyle (\de+2m-6)}  }.
\end{align*}
This time the injectivity criterion forces $l=4$, $m=0$, but again we obtain the same values for $s_1^\R,s_2^\R$ as in case 5c.

\subsection{Cases 6a, 6b, 6c:}

This is similar to cases 4a, 4b.   The formal solution 
has 
$\La_{-1} = P^{-1} (-x^2W^t)P = -x^2d_6$,  with 
$P = e^w \Om^{-1}$.  This time we have
$e^w=\diag(e^{w_0},e^{w_1},e^{w_2},e^{w_3},e^{w_4},e^{w_5}),
\Om=(\om^{ij})_{0\le i,j\le 5}$, $\om = e^{{2\pi \i}/6}$,  $\Pi=\Om d_6 \Om^{-1}$.

In Lemma \ref{formalsymmetries} we have
\[
 \Om \,\bar\Om^{-1} =
 \bp J_1 & \\
  & J_5
\ep,
 \ \ 
 \Om\De\Om=
 \begin{cases}
 6d_6^4 \ \text{in case 6a}\\
 6I \ \ \text{in case 6b}\\
 6d_6^2 \ \text{in case 6c}
 \end{cases}
\]
The matrix $(  \ar(\om^i-\om^j)  )_{0\le i,j\le 4}$ is
\[
\bp
\vphantom{\frac{5^5_5}{4^4_4}}  
0 &\tfrac{-2\pi}{6} & \tfrac{-\pi}{6} & 0 & \tfrac{\pi}{6} & \tfrac{2\pi}{6}\\
\vphantom{\frac{5^5_5}{4^4_4}}
\tfrac{4\pi}{6} & 0 & 0 & \tfrac{\pi}{6} & \tfrac{2\pi}{6} & \tfrac{3\pi}{6}\\
\vphantom{\frac{5^5_5}{4^4_4}}
\tfrac{5\pi}{6} & \tfrac{6\pi}{6} & 0 & \tfrac{2\pi}{6} & \tfrac{3\pi}{6} & \tfrac{4\pi}{6}\\
\vphantom{\frac{5^5_5}{4^4_4}}
\tfrac{6\pi}{6} & \tfrac{7\pi}{6} & \tfrac{8\pi}{6} & 0 & \tfrac{4\pi}{6} & \tfrac{5\pi}{6}\\
\vphantom{\frac{5^5_5}{4^4_4}}
\tfrac{7\pi}{6} & \tfrac{8\pi}{6} & \tfrac{9\pi}{6} & \tfrac{10\pi}{6} & 0 & \tfrac{6\pi}{6}\\
\tfrac{8\pi}{6} & \tfrac{9\pi}{6} & \tfrac{10\pi}{6} & \tfrac{11\pi}{6} &
\tfrac{12\pi}{6} & 0
\ep.
\]
The Stokes rays are given by $\ar\ze \in \tfrac\pi6\Z$. 
We choose
\[
\Om_1=\{ \ze\in\C \st -\tfrac{\pi}{2}<\ar\ze < \tfrac{4\pi}{6}\}
\]
as initial Stokes sector, then
$
\Om_{k+1}=e^{k\pi \i} \Om_1, \ \ k\in\tfrac16\Z.
$

Versions of Lemmas \ref{psik} and \ref{qk} hold here with the following modifications:

$d_6^{-1}\Psi_{k+\frac26}(\om\ze)\Pi=\Psi_k(\ze)$ (cyclic) and
$\overline{  \Psi_{\frac{11}6-k}(\bar\ze)  } \Om\bar\Om^{-1} = \Psi_k(\ze)$ (reality) in
Lemma \ref{psik};  

$Q_{k+\frac26}=\Pi Q_k \Pi^{-1}$ (cyclic) and 
$Q_k=    
\Om \,\bar\Om^{-1}\,
\bar Q_{ {\scriptstyle\frac{10}6} - k}^{-1}\,
(\Om \,\bar\Om^{-1})^{-1}$ (reality)  in
Lemma \ref{qk}.

The analogue of Lemma \ref{triangular} is that, for $i\ne j$, 
the $(i,j)$ entry of $Q_k$ is $0$ if $\ar(\om^i-\om^j) \not\equiv 
\frac{(23-19k)\pi}{6}$ mod $2\pi$.    Thus the fundamental matrices 
$Q_{1}, Q_{1\scriptstyle\frac16}$ have the form
$Q_{1}=I+q_{20}E_{20}+q_{35}E_{35}$, 
$Q_{1\scriptstyle\frac16}=I+q_{10}E_{10}+q_{25}E_{25}+q_{34}E_{34}$.
The reality conditions are  $q_{20}+\bar q_{35}=0$, $q_{10}+\bar q_{34}=0$, $q_{25}+\bar q_{25}=0$. Let
us write
\[
s_1=q_{10},\ \ s_2=q_{20},\ \ s_3=q_{25}.
\]
Then the anti-symmetry conditions imply that $s_3=0$ and
\[
s_1=
\begin{cases}
\om^{2} s_1^{\R}\ \text{in case 6a}\\
\ \ \ s_1^{\R}\ \text{in case 6b}\\
\om\ s_1^{\R}\ \text{in case 6c}
\end{cases}
\ \ 
s_2=
\begin{cases}
\om^{4} s_2^{\R}\ \text{in case 6a}\\
\ \ \ s_2^{\R}\ \text{in case 6b}\\
\om^2 s_2^{\R}\ \text{in case 6c}
\end{cases}
\]
where $s_1^{\R},s_2^{\R}$ are real.  

The monodromy at $\infty$ is  
$(Q_1 Q_{1\scriptstyle\frac16}\Pi)^6$, and the
characteristic polynomial of $Q_1 Q_{1\scriptstyle\frac16}\Pi$
is
\[
\la^6  - s_1\la^5 - s_2 \la^4 + \bar s_2 \la^2 + \bar s_1 \la - 1. 
\] 
If $\la$ is a root of this polynomial,
then so is $\frac{1}{\om^2\la}$ in case 6a,  $\frac{1}{\la}$ in case 6b, and
$\frac{1}{\om^4\la}$ in case 6c.

\no{\bf Proof of Theorem B(iv):}   
As the eigenvalues of
$(Q_1 Q_{1\scriptstyle\frac16}\Pi)^6$
are
$e^{\scriptstyle \pi\i \ga_0}$, 
$e^{\scriptstyle \pi\i \ga_1}$, 
$e^{\scriptstyle \pi\i \ga_2}$, 
$e^{\scriptstyle \pi\i \ga_3}$, 
$e^{\scriptstyle \pi\i \ga_4}$, 
$e^{\scriptstyle \pi\i \ga_5}$, 
there exist integers $a,b,c,d,e,f$ such that
the eigenvalues of $Q_1 Q_{1\scriptstyle\frac16}\Pi$
are 
\[
\om^a e^{\scriptstyle \frac\pi6\i \ga_0},
\om^b e^{\scriptstyle \frac\pi6\i \ga_1},
\om^c e^{\scriptstyle \frac\pi6\i \ga_2},
\om^d e^{\scriptstyle \frac\pi6\i \ga_3},
\om^e e^{\scriptstyle \frac\pi6\i \ga_4},
\om^f e^{\scriptstyle \frac\pi6\i \ga_5}.
\]
To determine $a,b,c,d,e,f$ we consider each case separately.

\no{\em Case 6a:}      

From the symmetry $\la\mapsto \frac{1}{\om^2\la}$,  we have
$
(a,b,c,d,e,f)=(l,m,2\text{ or }5,-m-2,-l-2,5\text{ or }2)
$
for some $l,m\in \Z/6\Z$.   
We obtain
\begin{align*}
s_1^\R
&=\om^{-2}  s_1=
e^{  {\scriptscriptstyle -\frac{4\pi \i}{\six}}} s_1\\
&=
e^{  {\scriptscriptstyle \frac{\pi \i}{\six}}(\ga+2l-4)}+
e^{  {\scriptscriptstyle \frac{-\pi \i}{\six}}(\ga+2l-4)}+
e^{  {\scriptscriptstyle \frac{\pi \i}{\six}}(\de+2m-4)}+
e^{  {\scriptscriptstyle \frac{-\pi \i}{\six}}(\de+2m-4)}\\
&=
2\cos{  \tfrac{\pi }{6}  {\scriptstyle (\ga+2l-4)}  }+
2\cos{ \tfrac{\pi }{6}  {\scriptstyle (\de+2m-4)}  }
\end{align*}
and 
\begin{align*}
-s_2^\R
&=-\om^{-4}s_2 =-e^{  {\scriptscriptstyle -\frac{8\pi \i}{\six}}} s_2\\
&=1+
4\cos{  \tfrac{\pi }{6}  {\scriptstyle (\ga+2l-4)}  }
\cos{ \tfrac{\pi }{6}  {\scriptstyle (\de+2m-4)}  }.
\end{align*}
The analogue of Proposition \ref{injective} is that the map 
$(\ga,\de)\mapsto (s_1^\R,s_2^\R)$ is injective on the region 
$\ga\ge -2$, $\de\le 2$, $\ga-\de\le 2$.    

The map $\ga\mapsto \cos \tfrac{\pi }{6}  {\scriptstyle(\ga+2l-4)} $
is injective on $[-2,4]$ only for $l\equiv 0,3 \mod 6$. The map $\de\mapsto \cos \tfrac{\pi }{6}  {\scriptstyle(\ga+2m-4)} $
is injective on $[-4,2]$ only for $m\equiv 1,4 \mod 6$.  We conclude that $(l,m)=(0,1)$ or $(3,4)$. This gives the stated values for $s_1^\R,s_2^\R$.

\no{\em Case 6b:}           

Here 
$
(a,b,c,d,e,f)=(0\text{ or }3,l,m,3\text{ or }0,-m,-l)
$
for some $l,m\in \Z/6\Z$.   This gives
\begin{align*}
s_1^\R
&=  s_1\\
&=
e^{  {\scriptscriptstyle \frac{\pi \i}{\six}}(\ga+2l)}+
e^{  {\scriptscriptstyle \frac{-\pi \i}{\six}}(\ga+2l)}+
e^{  {\scriptscriptstyle \frac{\pi \i}{\six}}(\de+2m)}+
e^{  {\scriptscriptstyle \frac{-\pi \i}{\six}}(\de+2m)}\\
&=
2\cos{  \tfrac{\pi }{6}  {\scriptstyle (\ga+2l)}  }+
2\cos{ \tfrac{\pi }{6}  {\scriptstyle (\de+2m)}  }
\end{align*}
and 
\begin{align*}
-s_2^\R
&=-s_2 \\
&=1+
4\cos{  \tfrac{\pi }{6}  {\scriptstyle (\ga+2l)}  }
\cos{ \tfrac{\pi }{6}  {\scriptstyle (\de+2m)}  }.
\end{align*}
The map $\ga\mapsto \cos \tfrac{\pi }{6}  {\scriptstyle(\ga+2l)} $
is injective on $[-2,4]$ only for $l\equiv 1,4 \mod 6$. The map $\de\mapsto \cos \tfrac{\pi }{6}  {\scriptstyle(\ga+2m)} $
is injective on $[-4,2]$ only for $m\equiv 2,5 \mod 6$.  Hence $(l,m)=(1,2)$ or $(4,5)$. 
This leads to the same values for $s_1^\R,s_2^\R$ as in case 6a.

\no{\em Case 6c:}       

Here 
$
(a,b,c,d,e,f)=(l,1\text{ or }4,-l-4,-m-4,4\text{ or }1,m)
$
for some $l,m\in \Z/6\Z$.   
We obtain
\begin{align*}
s_1^\R
&=\om^{-1}  s_1=
e^{  {\scriptscriptstyle -\frac{2\pi \i}{\six}}} s_1\\
&=
e^{  {\scriptscriptstyle \frac{\pi \i}{\six}}(\ga+2l-2)}+
e^{  {\scriptscriptstyle \frac{-\pi \i}{\six}}(\ga+2l-2)}+
e^{  {\scriptscriptstyle \frac{\pi \i}{\six}}(\de+2m-2)}+
e^{  {\scriptscriptstyle \frac{-\pi \i}{\six}}(\de+2m-2)}\\
&=
2\cos{  \tfrac{\pi }{6}  {\scriptstyle (\ga+2l-2)}  }+
2\cos{ \tfrac{\pi }{6}  {\scriptstyle (\de+2m-2)}  }
\end{align*}
and 
\begin{align*}
-s_2^\R
&=-\om^{-2}s_2 =-e^{  {\scriptscriptstyle -\frac{4\pi \i}{\six}}} s_2\\
&=1+
4\cos{  \tfrac{\pi }{6}  {\scriptstyle (\ga+2l-2)}  }
\cos{ \tfrac{\pi }{6}  {\scriptstyle (\de+2m-2)}  }.
\end{align*}
The map $\ga\mapsto \cos \tfrac{\pi }{6}  {\scriptstyle(\ga+2l-2)} $
is injective on $[-2,4]$ only for $l\equiv 2,5 \mod 6$. The map $\de\mapsto \cos \tfrac{\pi }{6}  {\scriptstyle(\ga+2m-2)} $
is injective on $[-4,2]$ only for $m\equiv 0,3 \mod 6$.  We conclude that $(l,m)=(2,3)$ or $(5,0)$. 
Again this gives the same values for $s_1^\R,s_2^\R$ as in case 6a.

{\em

\noindent
Department of Mathematics\newline
Faculty of Science and Engineering\newline
Waseda University\newline
3-4-1 Okubo, Shinjuku, Tokyo 169-8555\newline
JAPAN

\noindent
Department of Mathematical Sciences\newline
Indiana University-Purdue University, Indianapolis\newline
402 N. Blackford St.\newline
Indianapolis, IN 46202-3267\newline
USA
   
\noindent
Taida Institute for Mathematical Sciences\newline
Center for Advanced Study in Theoretical Sciences  \newline
National Taiwan University \newline
Taipei 10617\newline
TAIWAN

}

\end{document}